\documentclass[11pt, reqno]{amsart}
\usepackage{amscd}
\usepackage{amsmath,amssymb,amsthm,yfonts,mathrsfs,hhline}
\usepackage{hyperref}
\usepackage[dvips]{graphics}
\usepackage[centering,margin=1.2in]{geometry}
\usepackage[matrix,arrow,curve]{xy}
\global\binoppenalty=10000

\newtheorem{theorem}{Theorem}[section]
\newtheorem{lemma}[theorem]{Lemma}
\newtheorem{prop}[theorem]{Proposition}
\newtheorem{cor}[theorem]{Corollary}
\newtheorem{conjecture}[theorem]{Conjecture}

\theoremstyle{definition}
\newtheorem{defn}[theorem]{Definition}
\newtheorem{remark}[theorem]{Remark}
\newtheorem{notation}[theorem]{Notation}

\numberwithin{equation}{section}

\def\beq{\begin{equation}}
\def\eeq{\end{equation}}
\def\beqs{\begin{equation*}}
\def\eeqs{\end{equation*}}

\newcommand{\nc}[1]{\{#1\}}                                                             
\newcommand{\hr}[1]{\left(#1\right)}                                                    
\newcommand{\ha}[1]{\left\langle#1\right\rangle}                                        
\newcommand{\hs}[1]{\left[#1\right]}                                                    
\newcommand{\hc}[1]{\left\{#1\right\}}                                                  
\newcommand{\br}[1]{\bigl(#1\bigr)}                                                     
\newcommand{\bc}[1]{\bigl\{#1\bigr\}}                                                   
\newcommand{\Br}[1]{\Bigl(#1\Bigr)}                                                     
\newcommand{\bbr}[1]{\biggl(#1\biggr)}                                                  

\newcommand{\sco}{,\ldots,}                                                             

\newcommand{\suml}[2]{\sum\limits_{{#1}}^{{#2}}}                                        
\newcommand{\sums}[1]{\sum\limits_{{#1}}}                                               

\newcommand{\longra}{\longrightarrow}

\newcommand{\hra}{\hookrightarrow}                                                      

\newcommand{\eqdef}{\stackrel{\mbox{\tiny\rm def}}{=}}                                  

\newcommand{\haa}[1]{\left\langle\left\langle#1\right\rangle\right\rangle}

\def\Bc{\mathcal B}
\def\bgt{\mathfrak b}

\def\C{\mathbb C}

\def\dgt{\mathfrak d}
\def\Dq{\mathcal D_q}

\def\Frac{\operatorname{Frac}}

\def\g{\mathfrak g}
\def\gl{\mathfrak{gl}}
\def\GS{\widetilde G}

\def\hgt{\mathfrak h}
\def\Hc{\mathcal H}
\def\Hq{\mathcal H_q}

\def\id{\operatorname{id}}
\def\Ic{\mathcal I}

\def\Jc{\mathcal J}

\def\lm{{}'l^-}

\def\ngt{\mathfrak n}
\def\Norm{\operatorname{Norm}}

\def\Oc{\mathcal O}
\def\OR{{}^R\Oq}
\def\OF{{}^F\Oq}

\def\sl{\mathfrak{sl}}

\def\Tbb{\mathbb T}

\def\Uq{U_q(\g)}

\def\wh{\widehat}

\newcommand{\Hom}{\text{Hom}}
\newcommand{\End}{\mathrm{End}}

\newcommand{\Ub}{U_{\geq0}}
\newcommand{\Ubm}{U_{\leq0}}

\newcommand{\bC}{\mathbb{C}}

\newcommand{\Spec}{\operatorname{Spec}}

\newcommand{\Ad}{\text{Ad}}
\newcommand{\ad}{\text{ad}}

\newcommand{\N}{\mathbb{N}}
\newcommand{\h}{\mathfrak{h}}

\newcommand{\U}{\mathcal{U}}

\def\mgt{\mathfrak{m}}

\def\wt{\mathrm{wt}}

\def\Oq{\mathcal{O}_q[G]}

\begin{document}

\title[Quantum groups and quantum tori]{Quantum groups, quantum tori, and the Grothendieck-Springer resolution}
\author{Gus Schrader, Alexander Shapiro}
\address{Department of Mathematics, University of California, Berkeley, CA 94720, USA}

\begin{abstract}
We construct an algebra embedding of the quantum group $U_q(\g)$ into a central extension of the quantum coordinate ring $\Oc_q[G^{w_0,w_0}/H]$ of the reduced big double Bruhat cell in $G$. This embedding factors through the Heisenberg double $\Hq$ of the quantum Borel subalgebra $U_{\geq0}$, which we relate to $\Oq$ via twisting by the longest element of the quantum Weyl group. Our construction is inspired by the Poisson geometry of the Grothendieck-Springer resolution studied in~\cite{EL07}, and the quantum Beilinson-Bernstein theorem investigated in~\cite{BK06} and~\cite{Tan05}.
\end{abstract}

\maketitle

\section*{Introduction}
A basic and much-studied problem in the theory of quantum groups concerns finding embeddings of them into certain simpler algebras, which often lead to insights into their ring-theoretic and representation-theoretic properties. A well-known example of such an embedding is provided by the Feigin homomorphisms \cite{Ber96, Rup14} from the positive part $U_q(\ngt_+)$ of a Drinfeld-Jimbo quantized enveloping algebra to a \emph{quantum torus algebra}. For each reduced decomposition $w_0=s_{i_1}\cdots s_{i_l}$ of the longest element of the Weyl group, one has an algebra embedding of $U_q(\ngt_+)$ into the algebra generated by variables $X^{\pm1}_{1},\ldots, X^{\pm1}_{l}$ subject to the $q$-commutativity relations $X_iX_j = q^{b_{ij}}X_jX_i$. Other examples of quantum groups which have been shown to admit similar embeddings into quantum torus algebras include the quantum coordinate ring $\Oq$ of a simple Lie group $G$, as well as the quantum coordinate rings $\mathcal{O}_q[G^{u,v}]$ of its double Bruhat cells constructed by Berenstein and Zelevinsky in \cite{BZ05} and shown to bear an explicit structure of a quantum cluster algebra in~\cite{GY13}. These realizations of quantum groups are closely connected with the theory of quantum cluster ensembles~\cite{FG09}, the Feigin homomorphism playing the role of quantum factorization parameters, and the Berenstein-Zelevinsky realizations playing the role of generalized minors.

The problem of embedding the full quantized enveloping algebra $U_q(\g)$ into a quantum torus appears to be more subtle than the previous examples. In the construction of principal series representations for quantized enveloping algebras in~\cite{FI14, Ip12, Ip15}, homomorphisms from a certain modular double of $U_q(\g)$ to a quantum torus were obtained by explicitly writing formulas for the images of the Chevalley generators, and verifying by direct computation that the defining relations were satisfied. In particular, this method depends intricately upon the Dynkin type of $\g$. When $\g$ is of type $A$, quantum torus embeddings can also be obtained from representations of $\Uq$ of Gelfand-Zetlin type \cite{MT00}, which led to the proof of the Gelfand-Kirillov conjecture in \cite{FH14}. These embeddings, too, are constructed by explicitly verifying the relations in the Chevalley-Serre presentation of $U_q(\g)$. Subsequently, analogs of such representations for any quantum affine Kac-Moody algebra $\Uq$ were proposed in \cite{GKLO05}.

In this paper, we present a new approach to the construction of quantum torus realizations of $U_q(\g)$. Our construction is geometrically motivated, and requires no calculations with generators and relations, yet the homomorphism we obtain can be explicitly computed, see Corollary~\ref{phi-prime}. Our strategy is to construct an algebra embedding of $U_q(\g)$ into the algebra $\mathcal{O}_q[G^{w_0,w_0}/H]\otimes T$, where $\mathcal{O}_q[G^{w_0,w_0}/H]$ is the quantum coordinate ring of the big double Bruhat cell in $G$ reduced by the maximal torus $H$, and $T$ is the commutative torus subalgebra of $U_q(\g)$.  The desired embedding of $U_q(\g)$ into a quantum torus algebra can then be obtained using either the Feigin homomorphisms or the quantum torus realization of $\mathcal{O}_q[G^{w_0,w_0}/H]$ constructed in \cite{BZ05}. 

The geometric motivation for our construction comes from the {\em Grothendieck-Springer resolution} of the complex simple Lie group $G$. Recall \cite{BZN12} that the Grothendieck-Springer resolution for the Lie algebra $\g$ can be regarded as the moment map
$$
\mu_0\colon T^*(\widetilde{\mathcal{B}})/H\longrightarrow \g
$$
for the Hamiltonian action of $G$ on $T^*(\widetilde{\mathcal{B}})/H$, the quotient of the cotangent bundle of the base affine space $\widetilde{\mathcal{B}}=G/N$ by the maximal torus $H\subset G$. In particular, the resolution map $\mu_0$ is Poisson, where $\g$ carries the Kirillov-Kostant-Souriau Poisson structure. Quantizing the resolution yields an embedding of the enveloping algebra $U(\g)$ into the ring of $H$-invariant global differential operators on $\widetilde{\mathcal{B}}$, which is a key ingredient in the construction of the celebrated Beilinson-Bernstein equivalence of categories.  Moreover, by restricting to the open Schubert cell in $\widetilde{\mathcal{B}}$ one obtains the familiar realization of $U(\g)$ in terms of the Weyl algebra of differential operators on the big cell.

In \cite{EL07}, building upon the fundamental work of Semenov-Tian-Shansky \cite{STS85,STS92}, it was shown that there also exists a Poisson geometric interpretation of the multiplicative Grothendieck-Springer resolution
$$
\mu: X\longrightarrow G
$$
where $X$ is the variety consisting of pairs $(g,B')$ of an element $g\in G$ and a Borel subgroup $B' \subset G$ containing $g$, and the map $\mu$ is the projection forgetting $B'$.  An important result of \cite{EL07} is that $X$ and $G$ can be equipped with non-trivial Poisson structures in such a way that the resolution map $\mu$ becomes Poisson. The Poisson structure on $G$ may be regarded as the semiclassical limit of the quantum group $U_q(\g)$, or more precisely its $\ad$-integrable part $F_l(U_q(\g))$. It is thus natural to expect that quantizing the resolution $\mu$ would yield an interesting realization of $U_q(\g)$, just as its degeneration $\mu_0$ did for $U(\g)$.

We have endeavored to construct such a realization in the following fashion.  Using quantum Hamiltonian reduction, we define an algebra $\bC_q[X]$ which plays the role of the quantized algebra of global functions on $X$, together with an algebra embedding of $F_l(U_q(\g))$ into $\bC_q[X]$. This is closely related to the approach of \cite{BK06, BK11} and \cite{Tan05, Tan12, Tan14} to the quantum Beilinson-Bernstein equivalence.  Our next step is to find an appropriate analog of the weakly $H$-equivariant differential operators on the big cell of the basic affine space $\widetilde{\mathcal{B}}$. This role is played by the algebra of torus invariants in the Heisenberg double of $U_q(\mathfrak{b}_+)$, which we denote by $\mathcal{H}_q^{T_-}$. Then, using a certain twisting by the longest element $T_{w_0}$ of the quantum Weyl group of $U_q(\g)$, we construct an isomorphism of $\mathcal{H}_q^{T-}$ with $\mathcal{O}_q[G^{w_0}/H]\otimes T$, where $\mathcal{O}_q[G^{w_0}/H]$ is the quantum coordinate ring of the reduced big Bruhat cell in $G$. This gives us an embedding of $F_l(\Uq)$ into $\mathcal{O}_q[G^{w_0}/H]\otimes T$, and induces an embedding of the reduction of $F_l(\Uq)$ by any central character into $\mathcal{O}_q[G^{w_0}/H]$. Finally, in order to extend our algebra embedding from the ad-integrable part $F_l(U_q(\g))$ to all of $U_q(\g)$, we must further localize the target to obtain $\Oc_q[G^{w_0,w_0}/H]\otimes T$, where $\Oc_q[G^{w_0,w_0}/H]$ is the quantized algebra of functions on the reduced big double Bruhat cell in $G$.

We end our introduction by outlining some directions that we do not pursue here but which we plan to address in future work.  Firstly, it seems useful to explicitly study the embedding we construct using the quantum cluster coordinates on $\Oc_q[G^{w_0,w_0}/H]$ from \cite{Ber96, BZ05}. In particular, it would be interesting to understand the significance of quantum cluster mutations for $U_q(\g)$. Such an explicit coordinatization would also allow to compare our embedding to the ones of~\cite{GKLO05, FI14}. Finally, in a recent work~\cite{GSV15} a regular cluster structure on the semiclassical limit of $F_l(U_q(\gl_n))$ was constructed, while in \cite{Bra10} a different set of log-canonical coordinates, which do not extend to global regular functions, was proposed. We plan to study the connection between the seemingly different approaches to constructing cluster coordinates on quantum groups suggested by the present article, \cite{GSV15}, and~\cite{FH14} in a separate publication.

The article is organized as follows. In Section 1, we recall the Poisson geometry of the Grothendieck-Springer resolution of $G$, mostly following \cite{EL07}. In Section 2, we provide a short phrase-book between the main objects in the quantum part of the paper and their Poisson counterparts. Section 3 contains some definitions and standard facts regarding quantum groups that we use extensively throughout the paper. In section 4, we take a minor detour from the main objective of our paper, and discuss the quantized algebra $\C_q[X]$ of global functions on the Grothendieck-Springer resolution. While, strictly speaking, we do not require this algebra itself to obtain our main result, we nonetheless consider it an interesting intermediate step in analogy with the Poisson geometric picture outlined in Section 1. It also relates our work to the various approaches to the quantum Beilinson-Bernstein localization theorem, see~\cite{BK06, BK11} and~\cite{Tan05, Tan12, Tan14}. In Section 5, we embed the $\ad$-integrable part of $\Uq$ into the Heisenberg double $\Hq$ of the quantum Borel subalgebra $U_{\geq0}$. In Section 6, we show that certain localization of $\Oq$ is isomorphic to the subalgebra of $T$-invariants in $\Hq$. Section 7 contains our main result, namely, the algebra embedding $U_q(\g) \rightarrow \mathcal{O}_q[G^{w_0,w_0}/H]\otimes T$. Finally, in Section 8, we provide a detailed example of all our constructions in the case $\g=\sl_2$.

\section*{Acknowledgements}

We are indebted to Vladimir Fock for drawing our attention to the problem of constructing cluster coordinates on quantized enveloping algebras, which served as a primary motivation for this work. We are grateful to our advisor, Nicolai Reshetikhin, for his interest and support throughout the project. We also wish to thank David Ben-Zvi, Arkady Berenstein, Alexandru Chirvasitu, Iordan Ganev, David Jordan, and Dmitry Kubrak for many useful comments and discussions. We express our gratitude to the Center for Quantum Geometry of Moduli Spaces in Aarhus, where a part of this work was completed, for their generous hospitality. The second author was partially supported by the NSF grant DGE-1106400 and by the RFBR grant 14-01-00547.

\section{Poisson geometry}

In this section we recall the construction of the Grothendieck-Springer simultaneous resolution and its Poisson geometry. Our exposition mainly follows~\cite{STS85,EL07}.

\subsection{Conventions}

Throughout this section we will use the following conventions. Let $G$ be a complex simple Lie group, $B=B_+$ and $B_-$ a fixed pair of opposite Borel subgroups, $N=N_+$ and $N_- $ their unipotent radicals, and $H = B/N$ the corresponding torus. We denote by $\g$, $\bgt$, $\ngt$, and $\hgt$ the associated Lie algebras. The root system and the set of positive roots of $\g$ are denoted by $\Pi$ and $\Pi_+$ respectively. The Weyl group $W = \Norm_G(H)/H$ acts naturally on the torus $H$. Let $\Bc = G/B$ be the flag variety, whose points we identify with Borel subgroups of $G$.

For any $X \in \wedge^k\g$ we denote by $X^L$, $X^R$ the corresponding left and right invariant $k$-vector fields that take values $X^L(e) = X^R(e) = X$ at the identity element $e\in G$.

\subsection{The standard Poisson-Lie structure and related constructions}
The various Poisson structures on $G$ that we shall consider can be conveniently described in terms of the group $D = G \times G$, which we call the {\em double} of $G$.  We write $\dgt = \g \oplus \g$ for its Lie algebra. Let $\ha{\,,\,}$ denote (a nonzero scalar multiple of) the Killing form on $\g$. Then the pairing
\beq
\label{pairing-d}
\haa{(x_1,y_1),(x_2,y_2)} = \ha{x_1,x_2} - \ha{y_1,y_2}, \qquad x_1, x_2, y_1, y_2 \in \g
\eeq
defines a non-degenerate, ad-invariant symmetric bilinear form on $\dgt$. The pairing~\eqref{pairing-d} thus gives rise to the Manin triple $(\g, \g_\Delta, \g^*)$, where
$$
\g_\Delta = \hc{(x,x) \,|\, x\in\g}
$$
and
$$
\g^* = \hc{(x_++y, x_--y) \,|\, x_\pm \in \ngt_\pm, y\in\hgt}.
$$

Let $\xi_i \in \g_\Delta$ and $\xi^i \in \g^*$ be a pair of dual bases satisfying $\haa{\xi_i, \xi^j} = \delta_{ij}$. Then the canonical tensor
\beq
\label{r_D}
r_D = \frac12 \sum_i \xi^i \wedge \xi_i \in \dgt \wedge \dgt
\eeq
gives rise to the bivector field
$$
\pi_D^- = r_D^R - r_D^L
$$
which equips the double $D$ with the structure of a Poisson-Lie group.

\begin{notation}
In what follows we denote the Poisson-Lie group $(D,\pi_D^-)$ by $D_-$.
\end{notation}

Let $G_\Delta$ and $G^*$ be the connected Lie subgroups of $D$ with Lie bialgebras $\g_\Delta$ and $\g^*$ respectively. They take form
$$
G_\Delta \simeq G = \hc{(g,g) \,|\, g\in G}
$$
and
$$
G^* = \hc{(u_+t,t^{-1}u_-) \,|\, u_\pm \in N_\pm, t\in H}.
$$
The bivector fields
$$
\pi_- = \pi_D^-|_{G_\Delta} \qquad \pi_-^* = -\pi_D^-|_{G^*}
$$
turn $(G,\pi_-)$ and $(G^*,\pi_-^*)$ into a dual pair of Poisson-Lie groups. Note that $(G,\pi_-)$ and $(G^*,\pi_-^*)$ are Poisson-Lie subgroups of $D$.

\begin{notation}
We abbreviate $(G,\pi_-)$ and $(G^*,\pi_-^*)$ by $G_-$ and $G^*$ respectively.
\end{notation}

We can give an equivalent definition of the Poisson-Lie group $G_-$ as follows. Let $\bgt_+, \bgt_- \subset \g$ be a pair of opposite Borel subalgebras. If $x_i \in \ngt_+$ and $x^i \in \ngt_-$ are dual bases satisfying $\ha{x_i, x^j} = \delta_{ij}$, the canonical tensor
$$
r = \frac12 \sum_i x_i \wedge x^i \in \g \wedge \g.
$$
is called the \emph{standard $r$-matrix} for $\g$. The bivector $\pi_-$ can be written as
$$
\pi_- = r^R - r^L.
$$

Another Poisson structure on $G$ crucial for the sequel is given by the bivector field
$$
\pi_+ = r^R + r^L.
$$
Unlike $\pi_-$, $\pi_+$ is not a multiplicative Poisson structure. Therefore while $(G,\pi_+)$ carries the structure of a Poisson variety, it is not a Poisson-Lie group.

\begin{notation}
We abbreviate $(G, \pi_+)$ by $G_+$.
\end{notation}

Let $w_0 \in W$ be the longest element of the Weyl group of $G$.

\begin{lemma}
\label{prop-GG}
\begin{enumerate}
\item Inversion $\iota \colon g \mapsto g^{-1}$ is a Poisson automorphism of $G_+$;
\item Conjugation $\sigma \colon g \mapsto w_0gw_0^{-1}$ is an anti-Poisson automorphism of $G_-$;
\item The following two maps $G_- \longra G_+$
$$
g \mapsto gw_0 \qquad\text{and}\qquad g \mapsto w_0g
$$
are respectively Poisson and anti-Poisson isomorphisms.
\end{enumerate}
\end{lemma}

\begin{proof}
Point 1) follows from the fact that $\iota_*x^L = -x^R$ for any left invariant vector field $x^L$ on $G_+$. Points 2) and 3) follow from the identity $\sigma(r) = -r$.
\end{proof}

\subsection{Poisson reductions of the double}

Now let us consider the bivector field on $D$ defined by
$$
\pi_D^+ = r_D^R + r_D^L
$$
with $r_D$ given by~\eqref{r_D}. It defines on $D$ the structure of a Poisson variety (but not a Poisson-Lie group.)

\begin{notation}
\label{not-double}
In what follows we denote $(D, \pi_D^+)$ by $D_+$.
\end{notation}
The Poisson variety $D_+$ is often referred to as the {\em symplectic double} of $G$.
\begin{prop} \cite[Lemma 6.3]{EL07}
Consider the projections
$$
\begin{aligned}
\mu_1 \colon D \longra D/G_\Delta \simeq G, \qquad &(g_1,g_2) \mapsto g_1g_2^{-1}, \\
\mu_2 \colon D \longra G_\Delta \backslash D \simeq G, \qquad &(g_1,g_2) \mapsto g_2^{-1}g_1.
\end{aligned}
$$
Then there exist unique Poisson structures $\pi^*_1, \pi^*_2$ on $G$ such that the maps $\mu_1,\mu_2$ are Poisson.  Moreover, the Poisson structures $\pi^*_1, \pi^*_2$ are such that the maps
\begin{align}
\eta_1 \colon G^* \longra (G,\mu_1(\pi_D^+)), \qquad &(b_+, b_-) \mapsto b_+b_-^{-1}, \label{eta1} \\
\eta_2 \colon G^* \longra (G,\mu_2(\pi_D^+)), \qquad &(b_+, b_-) \mapsto b_-^{-1}b_+  \label{eta2}
\end{align}
are Poisson.
\end{prop}

\begin{notation}
In what follows we write $(G,\pi^*)$ for $(G,\pi_1^*)$.
\end{notation}

In what follows, we shall refer to $\mu_1$ and $\mu_2$ as moment maps. Indeed, it is immediate from the definitions that the action of $D_-$ by left/right multiplication on $D_+$ is Poisson. Let $D' = \mu_1^{-1}(B_+B_-)$ and $D'' = \mu_2^{-1}(B_-B_+)$. Then we have the following well-known result (see, e.g.~\cite[Example 3.3]{MW88}).

\begin{prop}
\label{prop-D-action}
The action of $G_\Delta \subset D_-$ on $D' \subset D_+$ by left multiplication admits a group-valued moment map $m_1 \colon D' \longra D'/G_\Delta \simeq G^*$ such that $\eta_1 \circ m_1 = \mu_1$. In a similar fashion, $\mu_2$ arises from the moment map for the right action of $G_\Delta$ on $D''$.
\end{prop}

Now, consider the Poisson action of $B_\Delta \subset D_-$ on $D_+$ by right multiplication.
\begin{prop}
\label{xquotient} \cite[Proposition 4.2]{EL07}
The subvariety $Q = \hc{(gb,g) \,|\, g\in G, b\in B} \subset D_+$ is coisotropic in $D_+$.  The quotient $Q/B_\Delta$ is a Poisson subvariety of $(D/B_\Delta, \phi(\pi_D^+))$ where $\phi \colon D \longra D/B_\Delta$ is the natural projection.
\end{prop}

We may identify $Q$ with the direct product $G \times B$ via
$$
Q \simeq G \times B, \qquad (g_1, g_2) \mapsto (g_2, g_2^{-1}g_1).
$$
The $B_\Delta$ action on $Q$ then reads
\beq
\label{action-B}
(G \times B) \times B \longra G \times B \qquad ((g,b), a) \mapsto (g a, a^{-1} b a).
\eeq
We denote the $B$-orbit through a point $(g,b)$ by $[g,b]$.
\begin{cor}
\label{cor-Poisson-X}
The set $X = G \times_B B$ of $B$-orbits under the action~\eqref{action-B} is a Poisson variety, with Poisson bivector $\pi_X = \phi(\pi_D^+)|_{Q/{B_\Delta}}$.
\end{cor}

Since the actions of $G_\Delta$ and $B_\Delta$ on $D_+$ by left and right multiplication respectively commute, the Poisson variety $X$ carries a residual Poisson action of $G_-$ given by
\beq
\label{action-G-X}
G \times X \longra X, \qquad (g', [g,b]) \mapsto [g'g,b].
\eeq
In view of Proposition~\ref{prop-D-action} we have

\begin{cor}
The action~\eqref{action-G-X} admits the following moment map
\beq
\label{moment-map}
\mu \colon X \longra G, \qquad [g,b] \mapsto gbg^{-1}
\eeq
with Poisson bivectors $\pi_X$ on the source and $\pi^*$ on the target.
\end{cor}

\begin{remark}
The map~\eqref{moment-map} was also shown in ~\cite{Boa11} to be the group-valued moment map in the sense of~\cite{AMM98}.
\end{remark}

\subsection{Grothendieck-Springer resolution}

Let $\GS = \nc{(g,B')\,|\, B'\in\Bc, g\in B'}$ be the set of pairs consisting of a Borel subgroup $B' \subset G$ and an element $g\in B'$. Equivalently, it may may be thought of as the set of pairs of a flag (preserved by $B'$) and an element $g$ preserving that flag. Since every element of $G$ is contained in a Borel, and all Borels are conjugate under $G$, we have

\begin{prop}
\label{prop-X-GS}
The map
\beq
\label{iso-GS}
\varpi \colon X \longra \GS, \qquad [g,b] \mapsto (gbg^{-1}, gBg^{-1})
\eeq
is an isomorphism of varieties.
\end{prop}

Let
$$
p \colon \GS \longra G, \qquad (g, B') \mapsto g
$$
be the projection onto the first factor. Note that the following diagram commutes
$$
\xymatrix{
X \ar[dr]_\mu \ar[r]^\varpi & \GS \ar[d]^p \\
& G
}
$$

\begin{defn}
The projection $p$ is called the \emph{Grothendieck-Springer (simultaneous) resolution}. Throughout the paper we will refer to the moment map~\eqref{moment-map} by the same name.
\end{defn}

Consider the map $\alpha \colon \GS \longra H$ defined by 
$\alpha([g,b]) = bN\in B/N\simeq H.$ Let $\beta\colon G \longra H/\!/W$ be the Chevalley restriction map,  coming from the inclusion $\C[H]^W \simeq \C[G]^G \hra \C[G]$. If we let $G_{reg}$ denote the locus of regular elements of $G$, and write $\GS_{reg} = p^{-1}(G_{reg})$, then the diagram
\beq
\label{cartesian-diag}
\xymatrix{
\GS_{reg} \ar[r]^\alpha \ar[d]_p & H \ar[d] \\
G_{reg} \ar[r]^\beta & H/\!/W
}
\eeq
is Cartesian.

\begin{cor} There is an isomorphism of coordinate rings
\beq
\label{W-inv}
\C[G] \simeq \C[X]^W.
\eeq
\end{cor}

\begin{proof}
Since the regular locus $G_{reg}$ is of codimension 3 in $G$, we have $\C[G_{reg}] = \C[G]$ and $\C[\GS_{reg}] = \C[\GS]$. In view of~\eqref{cartesian-diag}, we have an isomorphism
\beq
\label{cartesian}
\C[X] \simeq \C[G] \otimes_{\C[H]^W} \C[H].
\eeq
Taking $W$-invariants on both sides we obtain~\eqref{W-inv}.
\end{proof}

\subsection{From the standard Poisson structure to its dual}

Consider the homogeneous space $G/H = \hc{gH \,|\, g \in G}$.   For any $g\in B_+B_- \subset G$ define $[g]'_0 \in H$ to be the image of $g$ under the projection $G \longra N_+ \backslash G / N_-$. Since $\Ad_H(r)=r$, there is a unique Poisson tensor $\pi^+_{G/H}$ on $G/H$ with the property that the  projection $G \longra G/H$ is Poisson.

\begin{prop}
\label{prop-G/H}
The map
$$
D''_+ \longra G/H\times H, \qquad (g_1,g_2) \mapsto \hr{g_1H, \hs{g_2^{-1}g_1}'_0}
$$
is Poisson, where the Poisson structure on $G/H\times H$ is defined by the bivector $(\pi^+_{G/H},0)$.
\end{prop}

\begin{proof}
Let $\hc{y_i}$ be a basis for $\hgt$ such that $2\ha{y_i, y_j} = \delta_{ij}$. Then the bivectors $r$ and $r_D$ can be written as
$$
r = \frac12 \sums{\alpha\in\Pi_+} E_\alpha \wedge E_{-\alpha}
$$
and
$$
r_D = \frac12 \suml{i=1}{r} (y_i, -y_i) \wedge (y_i, y_i) + \frac12 \sums{\alpha\in\Pi_+}
\br{(E_\alpha, 0) \wedge (E_{-\alpha}, E_{-\alpha}) + (0, -E_{-\alpha}) \wedge (E_\alpha, E_\alpha)}
$$
respectively. Now it is clear that the bivector field $\pi_D^+$ maps to $\pi_+$ under the projection onto the first factor $(g_1,g_2) \mapsto g_1$.

Recall the moment map
$$
\mu_2 \colon D'' \longra G, \qquad (g_1,g_2) \mapsto g_2^{-1}g_1.
$$
It is easy to see that $\mu_2\hr{(x,x)^L} = 0$, $\mu_2\hr{(x,0)^R} = x^R$, and $\mu_2\hr{(0,x)^R} = -x^L$. Therefore,
$$
\mu_2(\pi_D^+) = \suml{i=1}{r} y_i^L \wedge y_i^R + \sums{\alpha\in\Pi_+} E_{-\alpha}^L \wedge E_\alpha^R
+ \frac12 \sums{\alpha\in\Pi_+} \hr{ E_\alpha^R \wedge E_{-\alpha}^R - E_{-\alpha}^L \wedge E_\alpha^L }.
$$
Further projection $G \longra N_+ \backslash G / N_-$ maps $E_\alpha^R$ and $E_{-\alpha}^L$ to 0, so $\mu_2(\pi_D^+)=\sum_{i=1}^{r} y_i^L \wedge y_i^R$. This bivector vanishes on the torus $H$ which finishes the proof.
\end{proof}

Denote by $G^\circ$ the open cell $B_-B_+ \subset G$. Any element $g\in G^\circ$ admits unique decomposition $g = [g]_-[g]_0[g]_+$ with $[g]_\pm \in N_\pm$ and $[g]_0 \in H$. Now, consider the projection
$$
X \longra G/B, \qquad [g,b] \mapsto gB,
$$
and let $X^\circ$ be the preimage of $G^\circ$.

\begin{prop}
The map
$$
\psi \colon G_+^\circ/H \times H \longra X^\circ, \qquad (u_-u_+H,t) \mapsto (u_-, u_+ t)B
$$
is an isomorphism of Poisson varieties.
\end{prop}

\begin{proof}
It is immediate that the inverse map is given by $[g,b]B \mapsto (gbH, [b]_0)$. It follows from Proposition~\ref{prop-G/H} that the isomorphism $\psi$ is Poisson.
\end{proof}

\begin{prop}
The map
$$
\tau \colon G_-/H \times H \longra G_+/H \times H, \qquad (gH,t) \mapsto (gw_0H,t).
$$
is an isomorphism of Poisson varieties.
\end{prop}

\begin{proof}
The bivector field $\pi_-$ vanishes on $H \subset G$, and thus descends to $G_-/H$. By Proposition~\ref{prop-GG} the map $G_- \longra G_+$ given by $g \mapsto gw_0$ is a Poisson isomorphism, which descends to a Poisson isomorphism $G_-/H\simeq G_+/H$ since $w_0H=Hw_0$.
\end{proof}

\begin{cor}
Let ${}^\circ G = B_- w_0 B_-$ be the big Bruhat cell in $G$. Then the map
$$
{}^\circ G_-/H \times H \longra (G,\pi^*) \qquad (gH,t) \mapsto gw_0t[gw_0]_-^{-1}
$$
is Poisson and coincides with the composition $\mu\circ\psi\circ\tau$.
\end{cor}

We conclude this section with a few remarks on the isomorphism $\psi$. Note that every element of $g \in G/H$ can be uniquely written as $g = u_-u_+H$ where $u_\pm \in N_\pm$. Consider the free right action of $H$ on $B_- \times B_+$ defined by
$$
(B_- \times B_+) \times H \longra (B_- \times B_+), \qquad \br{(b_-,b_+),t} \mapsto (b_-t, t^{-1}b_+t)
$$
and denote by $\Hc=(B_- \times B_+)/H$ the corresponding quotient. The following proposition shows that the isomorphism $\psi$ factors through $\Hc$. The proofs are immediate since each map has a well-defined inverse.

\begin{prop}
\label{birational-chain}
We have isomorphisms
\begin{align*}
G^\circ/H \times H \longra \Hc, &\qquad (u_-u_+H,t) \mapsto (u_-, u_+ t)H, \\
\Hc \longra X^\circ, &\qquad (b_-,b_+)H\mapsto [b_-,b_+]B.
\end{align*}
In particular, $\Hc$ is endowed with the structure of a Poisson variety.
\end{prop}

\section{A short guide to the rest of the paper}

In order to help the reader keep track of various quantum algebras that we use in the sequel, as well as their relation with the Poisson geometry outlined in Section 1, we provide here a brief phrase-book. The quantum algebras appearing in the table below may be regarded as quantizations of the algebras of global functions on the corresponding Poisson varieties. In general, the entries of the right column  are associative algebras, and carry the further structure of a Hopf algebra if their underlying Poisson variety is a Poisson-Lie group.

\begin{center}
\begin{tabular}{c||c}
Poisson varieties & Quantum algebras \\
\hhline{=::=}
$G_-$ & $\Oq$ \\
\hline
$G_+$ & $\OR$ \\
\hline
$\Hc$ & $\Hq^{T_-}$ \\
\hline
$X$ & $\C_q[X]$ \\
\hline
$G^*$ & $\Uq$ \\
\hline
$G_*$ & $F_l(U) \simeq \OF$ \\
\end{tabular}
\end{center}

The quantum algebras in the table above appear in the following order. In section~\ref{sub-quantum-groups} we introduce the quantum group $\Uq$. Its Hopf dual $\Oq$ is recalled in Sections~\ref{sub-coord} and~\ref{sub-quasi-triang}. Section~\ref{sub-l-op} explains relation between $\Oq$ and quantum Borel subalgebras $U_q(\bgt_\pm)$. In Section~\ref{sub-ad-fin} we introduce the ad-integrable part $F_l(U)$ of the quantum group $\Uq$ and explain its relation to the quantum coordinate ring. Section~\ref{sub-heis} is devoted to the Heisenberg double $\Hq$ of the quantum Borel subalgebra. The quantum analogue $\C_q[X]$ of the coordinate ring of the Grothendieck-Springer resolution is constructed and studied in Section~\ref{sect-GS}. Finally, the $R$-twisted quantum coordinate ring $\OR$ appears in Section~\ref{sect-R-twist}. At the quantum level, our construction is based on the interplay between the algebras $\Oq$, $\OR$, and $\OF$, which are all modelled on the same underlying vector space, but carry different associative algebra structure. Of these three algebras, only $\Oq$ is a Hopf algebra. This is completely parallel to the Poisson geometric picture, where we have three Poisson structures $G_-$, $G_+$, and $G_*$ on the same underlying variety $G$, with only $G_-$ being a Poisson-Lie group.

We summarize the various maps we construct between these quantum algebras in the diagram below:

$$
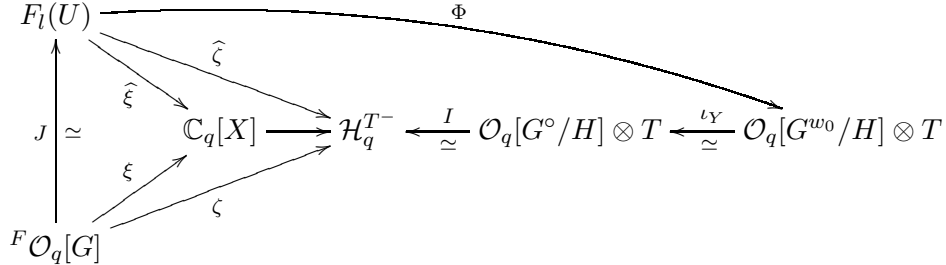
\xymatrix{
F_l(U) \ar[dr]_{\widehat\xi} \ar[drr]^{\widehat\zeta} \ar@/^1.5pc/[drrrr]^\Phi \\
& \C_q[X]\ar[r]& \Hq^{T^-}& \Oc_q[G^\circ/H]\otimes T  \ar[l]_{I\quad\;\;}^{\simeq\quad\;\;} & \Oc_q[G^{w_0}/H] \otimes T \ar[l]_{\;\;\iota_Y}^{\;\simeq} \\
\OF  \ar[uu]_{\simeq}^J \ar[ur]^{\xi}\ar[urr]_{\zeta}& & & 
}
$$
Let us remark that the maps $\widehat\xi$ and $\widehat\zeta$ are in fact restrictions of maps $\widetilde\xi$ and $\widetilde\zeta$ defined on $F_l(U)\otimes_Z U_0$, the extension of $F_l(U)$ over its Harish-Chandra center $Z$, see Section~\ref{Harishchandrasection}.  Finally, the map $\Phi$ can be extended to a map $\Phi'\colon \Uq\to \Oc'_q[G^{w_0,w_0}/H]\otimes T'$, where the algebra $\Oc'_q[G^{w_0,w_0}/H]\otimes T'$ is obtained from $\Oc_q[G^{w_0,w_0}/H]\otimes T$ by adjoining certain square roots, see Section~\ref{sect-main}.

\section{Preliminaries on quantum groups}
\label{prelim}

In this section, we recall the definitions and various well-known properties of the quantum groups that will be used extensively in the sequel.  Our conventions match those of of \cite{Jan96}. We refer the reader to \cite{Lus93, Jan96, KS97} for further details and proofs of many of the results in this section.

\subsection{Conventions}

In  what follows, $\g$ will denote a finite-dimensional complex simple Lie algebra of rank $r$, equipped with a choice of Cartan subalgebra $\h$ and a set of simple roots $\{\alpha_1,\ldots, \alpha_r\}$. We write $P,Q$ for the weight and root lattices associated to the corresponding root system $\Pi$, and denote the fundamental weights by $\omega_1,\ldots,\omega_r$.  Denote by $(\cdot,\cdot)$ the unique symmetric bilinear form on $\h^*$ invariant under the Weyl group $W$, such that $(\alpha,\alpha)=2$ for all short roots $\alpha\in \Pi$. Let $k=\C(q^{1/N})$ be the field of rational functions in a formal variable $q^{1/N}$, where $N\in \N$ is such that $\frac{1}{2}(\lambda,\mu)\in \frac{1}{N}\mathbb{Z}$ for any pair of weights $\lambda,\mu\in P$. If $A$ is a Hopf algebra, we denote by $A^{op}$ the Hopf algebra with the opposite multiplication to $A$, and denote by  $A^{cop}$ the Hopf algebra with the opposite comultiplication to $A$.  We will use the Sweedler notation
$$
\Delta(a)= \sum a_1\otimes a_2
$$
to express coproducts. Throughout the paper, all modules for the quantum group $U_q(\g)$ are assumed to be of type I.

\subsection{Quantized enveloping algebras}
\label{sub-quantum-groups}

The (simply-connected) quantized universal enveloping algebra $U \eqdef U_q(\g)$ is the $k$-algebra generated by elements $$\{E_i, F_i, K^\lambda  \ |\  i=1,\ldots, r, \lambda\in P\}$$ subject to the relations
$$
\begin{aligned}
&K^\lambda E_i =q^{(\lambda,\alpha_i)}E_i K^\lambda,
&\qquad K^{\lambda}K^{\mu} &=K^{\lambda+\mu}, \\
&K^\lambda F_i =q^{-(\lambda,\alpha_i)}F_i K^\lambda,
&\qquad [E_i,F_j] &=\delta_{ij}\frac{K_i-K_i^{-1}}{q_i-q_i^{-1}}.
\end{aligned}
$$
together with the quantum Serre relations (see \cite{Jan96}, p.53).  In the relations above we have set $K_i \eqdef K^{\alpha_i}$ and $q_i=q^{(\alpha_i,\alpha_i)/2}$. The algebra $U$ is a Hopf algebra, with the comultiplication
$$
\Delta(K^\lambda) = K^\lambda\otimes K^\lambda, \qquad
\Delta(E_i) = E_i\otimes 1+K_i\otimes E_i, \qquad
\Delta(F_i) = F_i\otimes K_i^{-1}+1\otimes F_i
$$
the antipode
$$
S(K^{\lambda})=K^{-\lambda}, \qquad S(E_i)=-K_i^{-1}E_i, \qquad S(F_i)=-F_iK_i
$$
and the counit
$$
\epsilon(K^{\lambda})=1, \qquad \epsilon(E_i)=0, \qquad \epsilon(F_i)=0.
$$

Let $\Ub$ denote the subalgebra of $U$ generated by all $K^\lambda, E_i$, and $\Ubm$ denote the subalgebra generated by all $K^\lambda, F_i$. We also write $U_0$ for the subalgebra generated by $K^\lambda, \; \lambda\in P$. The algebras $\Ub$, $\Ubm$, $U_0$ are Hopf subalgebras in $U$. Recall that $\hr{\Ubm}^{cop}$ stands for the co-opposite Hopf algebra to $\Ubm$. There is a non-degenerate Hopf pairing
\begin{align}
\label{qpair}
\ha{\cdot, \cdot} \colon \Ub \times \hr{\Ubm}^{cop} \longrightarrow k
\end{align}
defined by
$$
\langle K^{\lambda}, K^\mu \rangle = q^{-(\lambda,\mu)}, \qquad
\langle K^{\lambda}, E_i \rangle = 0 = \langle K^{\lambda}, F_i \rangle, \qquad
\langle E_i,F_j \rangle = -\frac{\delta_{ij}}{q_i-q_i^{-1}}.
$$

Let $U^+$ and $U^-$ denote the subalgebras generated by all $E_i$ and by all $F_i$ respectively. Then the quantum group $U$ admits a triangular decomposition: the natural multiplication map defines an isomorphism of $\C(q)$-modules
\beq
\label{triangulardecomp}
U^{+} \otimes U_0 \otimes U^{-} \longrightarrow U
\eeq
The algebra $U$ is graded by the root lattice $Q$. Indeed, setting
\beq
\label{q-grading}
U_\nu = \bc{ u\in U \;|\; K^\lambda u = q^{(\lambda,\nu)} uK^\lambda},
\eeq
we have $U = \bigoplus_{\nu\in Q}U_\nu$. If we set $U^+_\nu=U^+\cap U_\nu$ and $U^-_\nu=U^-\cap U_\nu$, then the pairing~\eqref{qpair} has the orthogonality property
\beq
\label{orthog}
\langle U^+_\nu, U^-_{-\mu} \rangle = 0 \quad \text{if} \quad \mu\neq \nu.
\eeq

\begin{remark}
The Hopf algebra $U$ can be described as a quotient of the Drinfeld double of the dual pair $(\Ub, \Ubm^{cop})$, which in particular implies the relation
\beq
\label{drinfeldmult}
xy = \langle x_1,y_1 \rangle \langle x_3,Sy_3 \rangle y_2x_2 \quad\text{for all}\quad x\in\Ub, \; y\in\Ubm.
\eeq
\end{remark}

\subsection{Quantized coordinate rings.}
\label{sub-coord}

Let $G$ be the connected, simply connected algebraic group with Lie algebra $\g$. The quantized algebra of functions on $G$, which we denote by $\Oq$, is defined to be the Hopf algebra of matrix elements of finite-dimensional $U$-modules. For a finite-dimensional $U$-module $V$ of highest weight $\lambda$ and a pair of elements $v\in V$ and $f\in V^*$ we denote the corresponding matrix element by $c^\lambda_{f,v}$, or simply by $c_{f,v}$ when it does not cause ambiguity. By construction, there is a Hopf pairing
\beq
\label{pairing}
\haa{\cdot,\cdot} \colon \Oq \otimes U \longrightarrow k
\eeq
defined by evaluation of matrix elements against elements of $U$. Pairing~\eqref{pairing} is non-degenerate, since no non-zero element of $U$ acts as zero in all finite-dimensional representations \cite{Jan96}.

The algebra $\Oq$ is a left $U\otimes U^{cop}$ module algebra via the left and right coregular actions
\begin{align}
\label{lractions}
((x\otimes y) \circ \psi)(u)=\psi(Sy u x) \qquad\text{where}\qquad x,u \in U, \ y \in U^{cop}, \ \psi \in \Oq.
\end{align}
As a $U\otimes U^{cop}$-module, $\Oq$ admits the Peter-Weyl decomposition
$$
\Oq=\bigoplus_{\lambda\in P^+}L(\lambda)^*\otimes L(\lambda)
$$
where $L(\lambda)$ is the finite-dimensional $U$-module of highest weight $\lambda$, and $L(\lambda)^*$ is its dual.

The algebra $\Oq$ is graded by two copies of the weight lattice $P$ as follows
$$
\label{oqgrading}
\Oq=\bigoplus_{\lambda,\mu\in P}\Oq_{\lambda,\mu}
$$
where
$$
\Oq_{\lambda,\mu}= \bc{\psi\in \Oq \;|\; (K^\nu \otimes K^\rho)\psi = q^{(\mu,\nu)+(\lambda,\rho)}\psi }.
$$
If $V$ is a representation of $U$ and $v\in V$ satisfies $K^\lambda v=q^{(\lambda,\mu)}v$ for all $\lambda\in P$, we say that $v$ is a weight vector of weight $\mu$, and write $\wt(v)=\mu$.   The subspace $\Oq_{\lambda,\mu}$ is spanned by matrix elements $c_{f,v}$ with $\wt(f)=\lambda$, $\wt(v)=\mu$. Note that $S(\Oq_{\lambda,\mu})=\Oq_{\mu,\lambda}$ and for $x_\nu\in U_\nu$, $\psi_{\lambda,\mu}\in \Oq_{\lambda,\mu}$ we have
$$
\psi_{\lambda,\mu}(x_\nu)\neq0 \ \ \implies \nu+\lambda+\mu=0
$$
Moreover, if $\psi\in \Oq_{\lambda,\mu}$ its coproduct takes the form
$$
\Delta(\psi)=\sum_\nu \psi_{\lambda,\nu}\otimes \psi_{-\nu,\mu} \quad\text{where}\quad \psi_{\alpha,\beta}\in \Oq_{\alpha,\beta}.
$$

\subsection{Quantum Weyl group}

Let $\widehat U$ be the completion of $U$ with respect to the weak topology generated by all matrix elements of finite-dimensional $U$-modules (see~\cite[Section 3]{KT08}). As an algebra $\widehat U$, is isomorphic to $\prod_{\lambda\in P_+} \End_{\C(q)}L(\lambda)$. We will also regard an element $u \in \widehat U$ as a functional on~$\Oq$ via the evaluation pairing $\haa{c_{f,v},u} = f(uv)$.

\begin{defn}\cite{Lus93}
Define an element $T_i$ of $\widehat U$ which acts on any weight vector $v$ by
$$
T_i(v) \;\;\; \eqdef \hspace{-1.5em} \sums{\substack{a,b,c \ge 0 \\ a-b+c = (\wt(v),\alpha_i)}} \hspace{-1em} (-1)^b q^{ac-b} F_i^{(a)}E_i^{(b)}F_i^{(c)}(v).
$$
\end{defn}

By~\cite[Theorem 39.4.3]{Lus93}, the elements $T_i$ generate an action of the braid group on any finite-dimensional $U$-module. The subalgebra of $\widehat U$ generated by $U$ together with the $T_i$ is often referred to as the {\em quantum Weyl group}, and it is known \cite{KR90} to in fact be a Hopf algebra. Moreover, let $w_0$ be the longest element of the Weyl group, and $w = s_{i_1} \dots s_{i_k}$ any of its reduced decompositions into simple reflections. Then the element $T_{w_0}$ defined by
\beq
\label{Tw0}
T_{w_0} = T_{i_1} \dots T_{i_k}
\eeq
is independent of the choice of reduced expression for $w_0$.

\subsection{Quantum minors}

We now recall the definition of certain elements of $\Oq$ that will prove useful in the sequel. 
For each dominant weight $\lambda\in P^+$, we fix a highest weight vector $v_\lambda\in L(\lambda)$.  Then, as in \cite{KT08}, we define the corresponding lowest weight vectors $v_{w_0(\lambda)}\in L(\lambda)$ by
$$
T_{w_0}v_{w_0(\lambda)} = (-1)^{\langle 2\lambda,\rho^\vee\rangle}q^{-2(\lambda,\rho)}v_{\lambda}
$$
\begin{prop}\cite[Comment 5.10]{KT08} The vectors $v_\lambda, v_{w_0(\lambda)}$ satisfy
$$
T_{w_0}v_{\lambda} = v_{w_0(\lambda)}.
$$
\end{prop}
For each $\lambda\in P^+$, there is a unique pairing 
$$
\langle -,-\rangle_{\lambda} \colon L(-w_0(\lambda)) \otimes L(\lambda) \longra k 
$$
satisfying conditions
$$
\ha{v_{-\lambda},v_\lambda}_\lambda = 1 \qquad\text{and}\qquad \ha{xw,v}_\lambda = \ha{w,Sxv}_\lambda
$$
for all $x \in U$, $v \in L(\lambda)$, and $w \in L(-w_0(\lambda))$. The following definition coincides with the one given in~\cite{BZ05}.

\begin{defn}
The {\em quantum principal minor} $\Delta^{\lambda}$ is the element of $\Oq$ whose value on any $x \in U$ is given by
$$
\Delta^{\lambda}(x)= \ha{v_{-\lambda},xv_\lambda}_\lambda
$$
Given $(u,v)\in W\times W$ we choose reduced decompositions $u=s_{i_l}\cdots s_{i_1}$ and $v= s_{j_{l'}}\cdots s_{j_1}$
and set
$$
n_k=\langle s_{i_1}\cdots s_{i_{k-1}}(\alpha_{i_k}^\vee),\lambda\rangle, \qquad
m_k=\langle s_{j_1}\cdots s_{j_{k-1}}(\alpha_{i_k}^\vee),\lambda\rangle.
$$
Then the quantum minor $\Delta^\lambda_{u,v}$ is defined by
$$
\Delta^\lambda_{u,v}(x) = \Delta^{\lambda}\left( E_{i_1}^{(n_1)}\cdots E_{i_l}^{(n_l)} x F^{(m_{l'})}_{j_{l'}}\cdots F_{j_1}^{(m_1)} \right)
$$
where $a^{(n)}$ stands for the $n$-th $q$-divided power of $a$. 
\end{defn}

\subsection{$\Oq$ as a co-quasitriangular Hopf algebra}
\label{sub-quasi-triang}

Write $\Theta_\nu$ for the canonical element in $U^+_\nu\otimes U^-_\nu$ with respect to the pairing~\ref{qpair}.  If $V, W$ are two finite-dimensional representations of $U$, then the action of the formal sum $\Theta=\sum_{\nu\in Q}\Theta_\nu$ is well defined in the tensor product $V\otimes W$.  Let $f_{V,W}$ be the operator in $V\otimes W$ defined by
$$
f_{V,W}(v\otimes w)=q^{-(\wt(v),\wt(w))}(v\otimes w)
$$
for any weight vectors $v,w\in V,W$.  Then define $R_{VW}$ to be the following operator in $V\otimes W$
$$
R_{VW}(v\otimes w)=\Theta\circ f_{VW}
$$
The operator $R$ gives rise to a bilinear form $r \colon \Oq\times \Oq\longrightarrow k$ defined by
\begin{align*}
r(c_{f,v},c_{g,w}) &= (f\otimes g) (R_{VW}(v\otimes w)) \\
&= \sum_{\alpha}q^{-(\wt(v),\wt(w))}f(\Theta_\alpha v)g(\Theta_{-\alpha}w).
\end{align*}
The form $r$ equips $\Oq$ with the structure of a {\em co-quasitriangular Hopf algebra} \cite{KS97, KS09}. This means that, for all triples $\phi,\psi,\rho\in \Oq$, we have
\begin{align}
\label{rform-prop}
&r(\phi_{1},\psi_{1})\phi_{2}\psi_{2} = \psi_{1}\phi_{1}r(\phi_{2},\psi_{2}), \\
&r(\phi\psi,\rho) = r(\phi,\rho_{1})r(\psi,\rho_{2}), \label{cop1} \\
&r(\rho,\phi\psi) = r(\rho_{1},\psi)r(\rho_{2},\phi). \label{cop2}
\end{align}
As the following Proposition shows, the form $r$ is closely related to the longest element $T_{w_0}$ of the quantum Weyl group.
\begin{prop} \cite{KR90,KT08} Let $C$ be the element of $\widehat U$ defined by
$$
C(v)=q^{(wt(v),\rho)-(wt(v),wt(v))/2}v
$$
where $\rho$ is the half-sum of positive roots. Then setting 
\beq
\label{y-element}
Y=CT_{w_0},
\eeq
 we have the following equality in $\Oq^* \otimes \Oq^*$
\beq
\label{cocycle}
r=(Y^{-1} \otimes Y^{-1})\Delta(Y).  
\eeq
\end{prop}

\subsection{$l$-operators}
\label{sub-l-op}

Let $\Oq^*$ be the full linear dual of $\Oq$, and define maps $$l^\pm,  \ {}'l^\pm \colon \Oq\longrightarrow\Oq^*$$ by
\begin{align*}
l^+(\phi) &=r(\cdot,\phi), &\quad 'l^+(\phi) &= r(\cdot,S^{-1}\phi), \\
'l^-(\phi) &=r(\phi,\cdot), &\quad l^-(\phi) &= r(S\phi,\cdot).
\end{align*}

\begin{lemma}\cite[Lemma 1.4]{KS09}
The maps $l^\pm: \Oq\rightarrow\Oq^*$ are anti-homomorphisms of algebras, while the maps $'l^\pm$ are homomorphisms of algebras.  Additionally, we have
$$
l^+, {}'l^+ \colon \Oq \longra \Ub, \qquad l^-, \lm \colon \Oq \longra \Ubm
$$
with explicit formulas given by
\begin{align*}
l^+(c_{f,v})&=\sum_\alpha f(\Theta_{-\alpha}v)\Theta_\alpha K^{-\wt(v)} \\
\lm(c_{f,v})&=\sum_\alpha f(\Theta_{\alpha}v)\Theta_{-\alpha} K^{-\wt(v)}
\end{align*}
\end{lemma}

We also have
\begin{lemma}
Let $\Delta$, $S_U$ denote the coproduct and antipode in $U$. Then
$$
'l^\pm=S_U\circ l^\pm
$$
and
\begin{align*}
&\Delta\circ l^\pm(\phi) = l^\pm(\phi_{1}) \otimes l^\pm(\phi_{2}), \\
&\Delta\circ {}'l^\pm(\phi) = {}'l^\pm(\phi_{2}) \otimes {}'l^\pm(\phi_{1}).
\end{align*}
\end{lemma}

\begin{proof}
These identities follow directly from the properties~\eqref{cop1}, \eqref{cop2} of $r$, together with the non-degeneracy of the Hopf pairing between $U$ and $\Oq$.
\end{proof}
We will make frequent use of the following lemma relating the universal $r$-form to the Hopf pairing~\eqref{qpair}.
\begin{lemma}
\label{compatible}
Let $\langle \cdot,\cdot\rangle$ be the pairing~\eqref{qpair} of $\Ub$ with $\Ubm$. Then
$$
\langle l^+(\phi),\ \lm(\psi)\rangle=r(\psi,\phi)
$$
\end{lemma}

\begin{proof}
We verify the claim for any pair of matrix elements $c_{f,v},c_{g,w} \in \Oq$.  Let us expand $\Theta=\sum_i\Theta_{+i}\otimes\Theta_{-i}$ where $\langle \Theta_{+i},\Theta_{-j}\rangle=\delta_{ij}$. Then using the relation
$$
\langle \Theta_{+i} K^{\lambda},\Theta_{-j} K^{\mu} \rangle=q^{-(\lambda,\mu)}\delta_{ij}
$$
from \cite[6.13]{Jan96} we compute
\begin{align*}
\langle l^+(c_{g,w}), \ ' l^-(c_{f,v})\rangle&=\sum_{i,j} g(\Theta_{-i}w)f(\Theta_{+j}v)\langle \Theta_{+i} K^{-\wt(w)},\Theta_{-j} K^{-\wt(v)} \rangle\\
&=\sum_{i}q^{-(\wt(v),\wt(w))}f(\Theta_{+i} v)g(\Theta_{-i}w) =r(c_{f,v},c_{g,w}).
\end{align*}
\end{proof}

\subsection{The $\ad$-integrable part of $U$}
\label{sub-ad-fin}

Consider the left (right) adjoint actions $\ad_l$ (respectively, $\ad_r$) of $U$ on itself defined by
\begin{align}
\ad_l(x)(y) &= x_1 y Sx_2 \label{adjoint} \\
\ad_r(x)(y) &= Sx_1 y x_2 \label{radjoint}
\end{align}

\begin{defn}
The \emph{left ad-integrable part of $U$} is defined as the subset
$$
F_l(U)=\{ x\in U \;|\; \dim \ad_l(U)x <\infty\}
$$
Similarly, the \emph{right ad-integrable part of $U$} is defined as the subset
$$
F_r(U)=\{ x\in U \;|\; \dim \ad_r(U)x <\infty\}
$$
\end{defn}

\begin{prop}\cite{KS97}
The ad-integrable parts $F_l(U)$, $F_r(U)$ are subalgebras in $U$. Moreover, they are left and right coideals respectively:
$$
\Delta(F_l(U))\subset U\otimes F_l(U), \qquad \Delta(F_r(U))\subset F_r(U)\otimes U.
$$
\end{prop}

Now consider the maps
\beq
\label{I}
I \colon \Oq \longrightarrow \Ub\otimes \Ubm, \qquad I= (l^+ \otimes \lm) \circ \Delta
\eeq
and
\beq
\label{j-map}
J \colon \Oq\longrightarrow U, \qquad J=m\circ I
\eeq
where
$$
m\colon\Ub\otimes\Ubm\longrightarrow U, \qquad u_+\otimes u_-\mapsto u_+u_-
$$
is the multiplication in $U$. Note also that the action~\eqref{radjoint} induces a coadjoint action $\ad_r^*\colon U\otimes \Oq \longrightarrow \Oq$ given by
\beq
\label{coadjoint}
\langle\ad^*_r(x)(\psi),y\rangle=\langle \psi, S(x_1)yx_2\rangle, \qquad x,y\in U, \ \psi\in\Oq.
\eeq
The following theorem was proven by Joseph and Letzter in \cite{JL94}, building on results of Caldero \cite{Cal93}.

\begin{theorem}\cite{JL94}
The map $J$ is an injection of $U$-modules, with respect to the action~\eqref{adjoint} on $U$ and the action~\eqref{coadjoint} on $\Oq$. Its image is
\begin{align}
\label{JLdecomp}
F_l(U)=\bigoplus_{\lambda\in P^+} (\ad_l U)(K^{-2\lambda})
\end{align}
\end{theorem}
Since $S(F_l(U))=F_r(U)$, the theorem implies that the map
$$
\label{tildej}
J' \eqdef S\circ J \colon \Oq \longrightarrow F_r(U)
$$
is also an isomorphism of $U$-modules. Indeed, for all $x\in U$, $\phi\in \Oq$ we have
\beq
\label{rintertwine}
x_2 J'(\phi)S^{-1}x_1 = J'(\ad^*_r(S^{-2}x)\phi)
\eeq
Despite being a morphism of $U$-modules, the map $J$ is not a morphism of algebras.  However, as explained in \cite{KS09}, one can equip $\Oq$ with a twisted algebra structure so that $J$ becomes an algebra homomorphism:

\begin{prop}
The following formula defines an associative product $\bullet_F$ in $\Oq$
\begin{align*}
\phi\bullet_F\psi&=r(\phi_{1}, \psi_{2} )r(\phi_{3},S\psi_{1})\phi_{2}\psi_{3}\\
&=r(\phi_2,\psi_3)r(\phi_3,S\psi_1)\psi_2\phi_1
\end{align*}
If we write $\OF$ for the algebra obtained by equipping $\Oq$ with the product $\bullet_F$, then the map $J \colon \OF\longrightarrow F_l(U)$ is an isomorphism of $U$-module algebras.
\end{prop}
Similarly, the map $J'$ is an isomorphism of algebras $(\OF)^{op}\simeq F_r(U)$.

\subsection{The Heisenberg double of $\Ub$}
\label{sub-heis}

We define the Heisenberg double of $U_{\geq0}$ to be the smash product $\Hq=\Ub\# \Ubm$ of the dual pair of Hopf algebras $U_{\geq0}$ and $U_{\leq0}^{cop}$ with respect to the pairing~\eqref{qpair}. The product in $\Hq$ can be written explicitly as
$$
(a\# x) (b\# y)=\langle b_{2},x_{2}\rangle ab_{1}\otimes x_{1}y
$$

Let us make a few remarks on the structure of $\Hq$ that will prove useful in the sequel. Consider the torus
$$
\Tbb=U_0\otimes U_0 \subset \Hq
$$
and the following three subtori
$$
T_+=U_0\otimes 1, \qquad T_-=1\otimes U_0, \qquad\text{and}\qquad T_{c}=(1\otimes S)\circ \Delta(U_0).
$$
The Heisenberg double $\Hq$ has the following $T_-$-module algebra structure
$$
(1\otimes K^\lambda) \circ (a\# x) = (1 \# K^\lambda) (a\# x) (1\# K^{-\lambda}) =\langle K^\lambda, a_2\rangle a_1\# K^\lambda x K^{-\lambda}.
$$
It also admits a $T_+$-module algebra structure given by
$$
(K^\mu\otimes 1)\circ (a\# x)=\langle K^\mu,x_1\rangle a\# x_2
$$
Since the actions of $T_+$ and $T_-$ commute, we may combine them into an action of $\Tbb$ on $\Hq$. Using the grading~\eqref{q-grading}, the restriction of this $\Tbb$-action to the subalgebras $T_-$ and $T_c$ can be computed explicitly as
\begin{align*}
(1\otimes K^{\rho})\circ(x_\nu K^\lambda \# y_\alpha K^\mu) &= q^{(\rho,\alpha-\lambda)}x_\nu K^\lambda \# y_\alpha K^\mu \\
(K^\rho\otimes K^{-\rho})\circ(x_\nu K^\lambda \# y_\alpha K^\mu) &= q^{(\rho,\lambda-\alpha-\mu)}x_\nu K^\lambda \# y_\alpha K^\mu
\end{align*}
for any $x_\nu\in U^+_\nu$ and $y_\alpha\in U^-_\alpha$. Therefore, we have

\begin{prop}
The $T_-$ invariants in $\Hq$ coincide with the subalgebra
$$
\Hq^{T_-}=\bigoplus_{\nu\in Q_+} U^+K^{-\nu}\# U^-_{-\nu }T,
$$
the $T^c$ invariants coincide with the subalgebra
\beq
\label{tcinvts}
\Hq^{T_c}=\bigoplus_{\lambda\in P, \, \nu\in Q_+} U^+K^{\lambda}\# U^-_{-\nu }K^{\lambda+\nu},
\eeq
and the $\Tbb$-invariants coincide with the subalgebra
$$
\Hq^{\Tbb}=\bigoplus_{ \lambda\in P, \, \nu\in Q_+} U^+K^{-\nu}\# U^-_{-\nu }.
$$
\end{prop}

Note that, the subalgebra of $T_-$ invariants $\Hq^{T_-}$ commute with the subalgebra $1\#T$. Hence we obtain
\begin{cor}
\label{heis-decomposition}
Multiplication in $\Hq$ yields an algebra isomorphism
$$
\Hq^{\Tbb} \otimes T \longrightarrow \Hq^{T^-}, \qquad (a\# x)\otimes (1\#K^\rho)\longmapsto a\# x K^\rho.
$$
\end{cor}

\begin{remark}
The torus $\Tbb(A)$ is naturally embedded into the Drinfeld double of the dual pair $(\Ub,\Ubm^{cop})$. The action of $\Tbb$ used in this section arises from the action of the Drinfeld double on the Heisenberg double considered in \cite{Lu94}.
\end{remark}

The following formula defines an action of $\Hq$ on $\Ub$
\begin{equation}
\label{heisrep}
(a \# x) \circ b = \langle x,b_{2}\rangle a b_{1}
\end{equation}
where $a\# x\in \Hq$ and $b\in \Ub$. We have the following lemma regarding the restriction of this action to the subalgebra $\Hq^{T_-}\subset \Hq$.

\begin{lemma}
\label{heisdecomp}
As $\Hq^{T_-}$-modules, we have
$$
\Ub=\bigoplus_{\lambda\in P} U^+K^\lambda
$$
\end{lemma}

\begin{proof}
It suffices to check that $K^{-\nu}\#y_{-\nu}\in \Hq^{T_-}$ preserves $U^+K^\lambda$. This follows from the orthogonality property~\eqref{orthog} of the pairing $\langle \cdot ,\cdot\rangle$ and the fact that given $x\in U^+_\alpha$, its coproduct can be expanded as
$$
\Delta(x)=\sum_{\beta} x_{\alpha-\beta}K^{\beta}\otimes x_{\beta}
$$
where $x_{\beta}\in U^+_\beta$ and $x_{\alpha-\beta}\in U^+_{\alpha-\beta}$.
\end{proof}

\section{Quantum Grothendieck-Springer resolution}
\label{sect-GS}

In this section, we describe an analog of the Grothendieck-Springer map~\eqref{moment-map} at the level of quantum groups.

\subsection{Quantum differential operators on $G$}

Following \cite{STS92, BK06}, we define the ring $\Dq$ of quantum differential operators on $G$, to be the smash product algebra
$$
\Dq=\Oq^{op}\# U^{cop}.
$$
The multiplication in $\Dq$ is given by the formula
$$
\phi\# u \cdot \psi\# v= \psi_2(u_2)\psi_1\phi\# u_1v
$$
Since $F_r(U)$ is a right coideal, the algebra $\Dq$ contains a subalgebra
\beq
\label{finitediffops}
\Dq^{fin} = \Oq^{op}\# F_r(U).
\eeq
We can parameterize $\Dq^{fin}$ using the algebra isomorphism $J' \colon (\OF)^{op} \longrightarrow F_r(U)$. Under this identification, one checks that product in $\Dq^{fin}$ becomes
\beq
\label{Jproduct}
(\phi\#J'(\rho))\cdot(\psi\#J'(\nu))
= r(S^{-1}\rho_3,\psi_2)r(S^{-1}\nu_1S^{-2}\nu_3\psi_3,S^{-1}\rho_1)\psi_1\phi\#J'(\nu_2\rho_2).
\eeq

\subsection{Construction of $\C_q[X]$}

We are now ready to describe our construction of $\C_q[X]$, the quantized algebra of global functions on the Grothendieck-Springer resolution.  The idea is to obtain $\C_q[X]$ as the quantum hamiltonian reduction of $\Dq$ under an appropriate action of $U_{\geq0}$, analogously to the construction of $X$ as the quotient of the coisotropic subvariety $Q\subset D_+$ by the Poisson action of $B_\Delta$ in Proposition~\ref{xquotient}.

\begin{remark}
Philosophically, the construction of $\C_q[X]$ presented in this section is very similar to the one of global sections of the sheaf of quantum differential operators in \cite{BK06, BK11}. However, there appear to be some technical differences between the two constructions, so we provide independent proofs of the results we shall use.
\end{remark}

For the reader's convenience, we briefly recall the notion of quantum Hamiltonian reduction.  Suppose that $H$ is a Hopf algebra, $A$ is an associative algebra, $\mu \colon H\rightarrow A$ is a homomorphism of associative algebras, and $I$ is a 2-sided ideal in $H$ preserved by the left adjoint action of $H$. Then by the $\ad$-invariance of $I$, the action of $H$ on $A$ defined by the formula
$$
h\circ a= \mu(h_1)a\mu(Sh_2)
$$
descends to an action of $H$ on the $A$-module $A/A\mu(I)$. The \emph{quantum Hamiltonian reduction} of $A$ by the \emph{quantum moment map} $\mu \colon H \rightarrow A$ at the ideal $I$ is defined as the set of $H$-invariants
$$
\hr{ A/A\mu(I) }^H = \{ a\in A/A\mu(I) \;|\; h\circ a=\epsilon(h)a \quad\text{for all}\quad h\in H \}
$$
which one checks inherits a well-defined associative algebra structure from that of $A$.

In order to obtain $\C_q[X]$ as quantum hamiltonian reduction of the ring $\Dq$, the above construction requires some technical modifications which we shall now explain. Consider the algebra embedding of $U$ into $\Dq=\Oq^{op}\# U^{cop}$ given by
$$
\mu \colon U \longrightarrow \Dq,  \qquad u\longmapsto 1\# u.
$$
Regarding this algebra embedding as a quantum moment map, one obtains the following $U^{cop}$-module algebra structure on $\Dq$
\beq
\label{Uaction}
x \circ (\phi \# u) = (1\# x_{2})(\phi \# u)(1\# S^{-1}x_{1}) =  \phi_2(x_3) \phi_1\# x_{2}uS^{-1}x_{1}.
\eeq
This action preserves the subalgebra $\Dq^{fin}\subset \Dq$ defined in~\eqref{finitediffops}, and restricts to $U_{\geq0}\subset U$ as follows \beq
\label{dqfinaction}
x\circ \phi\#J'(\rho) = (S^{-1}\rho_1S^{-2}\rho_3\phi_2)(x)\phi_1\#J'(\rho_2)
\eeq
where $x\in U$ and $\phi\#J'(\rho)\in \Dq^{fin}$.

Consider now the $\C(q)$-linear map
$$
\Dq^{fin} \longrightarrow \Oq^{cop} \otimes U_{\geq0}, \qquad \phi\#J'(\rho) \longmapsto \phi\# 'l^{-}(S^{-1}\rho)$$
and denote by $\Ic$ the kernel of this map.

\begin{prop}
\label{idealinv}
$\Ic$ is a left ideal in $\Dq^{fin}$, and $\Ic$ is preserved by the action~\eqref{Uaction} of~$U_{\geq0}$.
\end{prop}

\begin{proof}
Consider $\phi,\psi\in \Oq$ so that $\lm(\phi)=0$. Then for all $\rho\in \Oq$ we have
\begin{align*}
\langle \lm(\phi\bullet_F\psi),\rho\rangle&=r(\phi_2,\psi_3)r(\phi_3,S\psi_1)r(\psi_2\phi_1,\rho)
=r(\phi_2,\psi_3)r(\phi_3,S\psi_1)r(\psi_2,\rho_1)r(\phi_1,\rho_2) \\
&=r(\phi,S\psi_1\psi_3\rho_2)r(\psi_2,\rho_1)
=\langle \lm(\phi),S\psi_1\psi_3\rho_2\rangle r(\psi_2,\rho_1)
=0
\end{align*}
which implies $\lm(\phi\bullet_F\psi)=0$. Hence the assertion that $\Ic$ is a left ideal follows from formula~\eqref{Jproduct} for the product in $\Dq^{fin}$.

Let us  now show that $\Ic$ is preserved under the action~\eqref{Uaction} of $U_{\geq0}$. By formula~\eqref{rintertwine} it suffices to show that for all $x\in U_{\geq0}$ and $\phi\in \Oq$ such that $\lm(\phi)=0$, we have $\lm(\ad^*_r(S^{-2}x)\phi)=0$. Since $l^+$ is surjective, we may write $x=l^+(\nu)$ for some $\nu\in\Oq$. Then for all $\eta\in \Oq$, we get
\begin{align*}
\langle l^-(\ad^*_r(S^{-2}x)\phi),\eta\rangle&=\phi_3(S^{-2}x_2)\phi_1(S^{-1}x_1)\langle\lm(\phi_2),\eta\rangle\\
&=r(\phi_3,S^2\nu_2)r(\phi_1,S\nu_1)r(\phi_2,\eta)
=r(\phi,S^2\nu_2\eta S\nu_1)
=0
\end{align*}
which shows that $\lm(\ad^*_r(S^{-2}x)\phi)=0$, completing the proof.

\end{proof}

It follows from Proposition~\ref{idealinv} that the action~\eqref{Uaction} descends to a well-defined action on the quotient $\mathcal{D}^{fin}_q/\Ic$. We now define $\C_q[X]$ to be the set of $U_{\geq0}$-invariants in $\mathcal{D}^{fin}_q/\Ic$ with respect to action~\eqref{Uaction}
$$
\C_q[X] \eqdef \left(\mathcal{D}^{fin}_q/\Ic\right)^{U_{\geq0}}
$$

\begin{prop}
\label{prop-mult}
The formula
\beq
\label{cqprod}
\hr{z+\Ic} \cdot \hr{w+\Ic} = zw+\Ic, \qquad z+\Ic, \, w+\Ic\in\C_q[X]
\eeq
equips $\C_q[X]$ with a well-defined associative product.
\end{prop}

\begin{proof}
First, let us verify that~\eqref{cqprod} provides a well-defined map $\C_q[X]^{\otimes2}\rightarrow \mathcal{D}^{fin}_q/\Ic$. All we must check is that $\Ic$ is a two-sided ideal in $\mathcal D_q^{fin}$. It suffices to show that given a $U_{\geq0}$-invariant $z+\Ic \in \C_q[X]$ and an element $\nu\in\Oq$ satisfying $\lm(\nu)=0$ we have $(1\#J'(\nu))\cdot z\in \Ic$.

Let $z=\sum_k\phi^k\#J'(\rho^k)$, then the $U_{\geq0}$-invariance of $z$ is equivalent to the condition that for all $\psi\in\Oq$, we have
\beq
\begin{aligned}
\label{invcecondition}
(\id \otimes \lm) \bbr{l^+(\psi)\circ \bbr{\sum_k\phi^k\#\rho^k}}
&=\sum_k r(S^{-1}\rho^k_1S^{-2}\rho^k_3\phi^k_2,\psi)\phi^k_1\#\lm(\rho^k_2) \\
&=\epsilon(\psi)\sum_k\phi^k\#\lm(\rho^k).
\end{aligned}
\eeq
So by~\eqref{Jproduct} and the invariance condition~\eqref{invcecondition}, we get
\begin{align*}
(\mathrm{id}\otimes \lm)(1\# J'(\nu))\cdot (z+\Ic)&=\sum_k r(S^{-1}\nu_3,\phi^k_2)r(S^{-1}\rho^k_1S^{-2}\rho^k_3\phi^k_3,S^{-1}\nu_1)\phi^k_1\#\lm(\rho^k_2\nu_2)\\
&=\sum_kr(S^{-1}\nu_2,\phi^k_2)\phi^k_1\# l'^-(\rho^k\nu_1)
\end{align*}
This implies that for all $\eta\in \Oq$, we have
\begin{align*}
\langle (\mathrm{id}\otimes \lm)(1\#J'(\nu)\cdot z), \mathrm{id}\otimes\eta\rangle
&=\sum_kr(S^{-1}\nu_2,\phi^k_2)\langle \lm(\rho^k\nu_1),\eta\rangle\phi^k_1 \\
=\sum_kr(\nu_2,S\phi^k_2)r(\rho^k\nu_1,\eta)\phi^k_1
&=\sum_kr(\nu_2,S\phi^k_2)r(\rho^k,\eta_1)r(\nu_1,\eta_2)\phi^k_1 \\
&=\sum_kr(\nu,S\phi^k_2\eta_2)r(\rho^k,\eta_1)\phi^k_1
=0
\end{align*}
where $\lm(\nu)=0$ is used in the last equality. Thus, $(1\#J'(\nu))\cdot z\in \Ic$ as claimed.

To complete the proof that~\eqref{cqprod} is a well-defined product on $\C_q[X]$, we must check that
$$
z+\Ic, \, w+\Ic\in\C_q[X] \quad\text{implies}\quad zw+\Ic \in \C_q[X].
$$
Indeed, since~\eqref{Uaction} defines on $\Dq$ the structure of a $U^{cop}$-module algebra, it follows from the first part of the proof that
\begin{align*}
a\circ(zw+\Ic)&=a\circ(zw)+\Ic
=(a_2\circ z)(a_1\circ w)+\Ic
=(\epsilon(a_2)z+\Ic)(\epsilon(a_1)w+\Ic)
=\epsilon(a)zw+\Ic
\end{align*}
which completes the verification that $\C_q[X]$ is an algebra.

\end{proof}

Since $\lm \colon \Oq \rightarrow U_{\leq0}$ is surjective, we have an identification of $\C(q)$-modules
\beq
\label{cqmodel}
\mathcal{D}^{fin}_q/\Ic \simeq \Oq^{op}\otimes U_{\leq0}, \qquad \phi\#J'(\rho)+\Ic \longmapsto \phi\# \lm(S^{-1}(\rho))
\eeq

\begin{prop}
\label{alg-str}
The algebra structure of $\C_q[X]=\left(\Oq^{op}\otimes U_{\leq0}\right)^{U_{\geq0}}$ is given by
\begin{align}
\label{explicitcqprod}
(\phi\#x)\cdot(\psi\#y)=\psi_2(x_2)\psi_1\phi\#x_1y
\end{align}
Thus $\C_q[X]$ may be regarded as a subalgebra in the smash product $\Oq^{op}\#U_{\leq0}^{cop}\subset \Dq$.
\end{prop}

\begin{proof}
Choose elements $\rho,\nu\in \Oq$ so that $\lm(S^{-1}\rho)=x$, and $\lm(S^{-1}\nu)=y$. The $U_{\geq0}$ invariance condition~\eqref{invcecondition} for $\psi\#J'(\nu)+\Ic$ implies
\begin{align*}
(\id \otimes \lm\circ S^{-1}) \hr{ (\phi\#J'(\rho))\cdot(\psi\#J'(\nu)) }
= r(S^{-1}\rho_2,\psi_2)\psi_1\phi\#\lm(S^{-1}(\nu\rho_1)) \\
= \psi_2(\lm(S^{-1}\rho_2)\psi_1\phi\#\lm(S^{-1}\rho_1)\lm(S^{-1}\nu)
= \psi_2(x_2)\psi_1\phi\#x_1y.
\end{align*}
\end{proof}

\begin{cor}
\label{cqtoheis}
The map
\begin{align}
l^+ \otimes \id \colon \hr{\Oq^{op} \otimes U_{\leq0}}^{U_{\geq0}}\simeq \C_q[X]\longrightarrow \Hq
\end{align}
is a homomorphism of algebras.
\end{cor}

\begin{proof}

Suppose that $\phi\#x, \psi\#y\in \C_q[X]$ with $x=\lm(\rho)$. Then
\begin{multline*}
(l^+(\phi)\#x)\cdot (l^+(\psi)\#y)
=\langle x_2,l^+(\psi)_2\rangle l^+(\phi)l^+(\psi_1)\#x_1y
=\langle x_2,l^+(\psi)_2\rangle l^+(\phi)l^+(\psi_1)\#x_1y\\
=\langle \lm(\rho_1),l^+(\psi_2)\rangle l^+(\phi)l^+(\psi_1)\#\lm(\rho_2)y
=r(\rho_1,\psi_2)l^+(\psi_1\phi)\#\lm(\rho_2)y\\
=\psi_2(\lm(\rho_1))l^+(\psi_1\phi)\#\lm(\rho_2)y
=\psi_2(x_2)l^+(\psi_1\phi)\#x_1y
=(l^+\otimes \mathrm{id})((\phi\#x)\cdot(\psi\#y)).
\end{multline*}
\end{proof}

\subsection{Construction of the resolution.}

Consider the map
$$
\varrho \colon \OF \longrightarrow \mathcal{D}^{fin}_q, \qquad \phi \longmapsto S^{-1}\phi_{3}\phi_{1} \# \phi_{2}
$$

\begin{prop}
One has $\varrho(\OF)\subset \left( \mathcal{D}^{fin}_q \right)^{U}$, where $U$ acts via~\eqref{Uaction}.
\end{prop}

\begin{proof}
Let $x\in U$ and $\phi\in\OF$. Then by~\eqref{dqfinaction} we get
$$
x\circ\varrho(\phi)
=(S^{-1}\phi_3S^{-2}\phi_5 S^{-1}\phi_6\phi_2)(x) \ S^{-1}\phi_7\phi_1\#\phi_4
=\epsilon(x)S^{-1}\phi_3\phi_1\#\phi_1
=\epsilon(x)\varrho(\phi).
$$
\end{proof}

\begin{cor}
The image of the natural map $\xi \colon \OF \longrightarrow \mathcal{D}^{fin}_q/\Ic$ obtained by composing $\varrho$ with the quotient projection is contained in $\C_q[X]=\left( \mathcal{D}^{fin}_q/\Ic\right)^{U_{\geq0}}$.
\end{cor}
Identifying $\C_q[X]$ with $\Oq^{op}\# U_{\geq0}$ via~\eqref{cqmodel}, we get
$$
\xi(\phi)=S^{-1}\phi_3\phi_1\#\lm(S^{-1}\phi_2)
$$

\begin{theorem}
\label{thm-res}
The map $\xi \colon \OF \longrightarrow \C_q[X]$ is a homomorphism of algebras.
\end{theorem}

\begin{proof}
We have
\begin{align*}
\xi(\phi\bullet_F\psi)&=r(\phi_2,\psi_3)r(\phi_3,S\psi_1)\xi(\psi_2\phi_1)\\
&=r(\phi_4,\psi_5)r(\phi_5,S\psi_1)S^{-1}(\phi_3)S^{-1}(\psi_4)\psi_2\phi_1\# \lm(S^{-1}(\psi_3\phi_2))\\
&=r(\phi_3,\psi_4)r(S^{-1}\phi_5,\psi_1)S^{-1}(\psi_5)S^{-1}(\phi_4)\psi_2\phi_1\# \lm(S^{-1}(\psi_3\phi_2))\\
&=r(\phi_3,\psi_4)r(S^{-1}\phi_4,\psi_2)S^{-1}(\psi_5)\psi_1 S^{-1}(\phi_5)\phi_1\# \lm(S^{-1}(\psi_3\phi_2)).
\end{align*}
On the other hand, using formula~\eqref{explicitcqprod} for the product in $\C_q[X]$ we have
\begin{align*}
\xi(\phi)\cdot\xi(\psi)&=S^{-1}\phi_3\phi_1\# \lm(S^{-1}\phi_2)\bullet S^{-1}\psi_3\psi_1\# \lm(S^{-1}\psi_2)\\
&=\langle S^{-1}(\psi_4) \psi_2, \lm(S^{-1}\phi_3) \rangle S^{-1}\psi_5\psi_1 S^{-1}(\phi_4)\phi_1\# \lm(S^{-1}(\psi_3\phi_2))\\
&= r(S^{-1}\phi_3,S^{-1}(\psi_4) \psi_2)  S^{-1}\psi_5\psi_1 S^{-1}(\phi_4)\phi_1\# \lm(S^{-1}(\psi_3\phi_2))\\
&=r(\phi_3,\psi_4)r(S^{-1}\phi_4,\psi_2)S^{-1}(\psi_5)\psi_1 S^{-1}(\phi_5)\phi_1\# \lm(S^{-1}(\psi_3\phi_2))\\
&=\xi(\phi\bullet_F\psi).
\end{align*}
\end{proof}

Precomposing with the isomorphism $J^{-1} \colon F_l(U) \rightarrow \OF$ defined in~\eqref{j-map} we obtain

\begin{cor}
The map
$$
\widehat{\xi} \colon F_l(U) \longrightarrow \C_q[X], \qquad \widehat{\xi}=\xi\circ J^{-1}
$$
is a homomorphism of algebras.
\end{cor}

\subsection{Restriction of $\widehat\xi$ to the center of $U$.}
\label{Harishchandrasection}

The homomorphism $\widehat{\xi}$ bears an interesting relation to the center $Z$ of $U$. This center can be described in several ways.  Firstly, we have the quantum Harish-Chandra map $\vartheta \colon U \longrightarrow U_0$, which is defined in terms of the triangular decomposition~\eqref{triangulardecomp} of $U$:
$$
\vartheta \colon U\simeq U_{>0}\otimes U_0\otimes U_{<0} \longrightarrow U_0, \qquad a\otimes t\otimes x\longmapsto \epsilon(a)\epsilon(x)t
$$
Tthe restriction of $\vartheta$ to $Z\subset U$ is an injective algebra homomorphism \cite{Jan96}. To describe its image, let
$$
\rho=\frac{1}{2}\sum_{\alpha\in \Pi_+}\alpha
$$
be the half-sum of positive roots, and consider the $\C(q)$-algebra automorphism $\kappa \colon U_0 \longrightarrow U_0$ defined by
$$
\kappa(K^\lambda)=q^{(\rho,\lambda)}K^\lambda
$$
The Weyl group $W$ acts on $U_0$ by
$$
w\cdot K^\lambda=K^{w(\lambda)}
$$
and we denote by $U_0^{W}$ its fixed point subalgebra. Finally, we write
\beq
\label{evenpart}
U_{0,\,even}=\bigoplus_{\lambda\in P}K^{2\lambda}
\eeq
We have the following quantum analog of Harish-Chandra's theorem:
\begin{prop}\cite[Chapter 6]{Jan96}
The map $\vartheta \colon Z \longrightarrow \kappa \left(U_{0,\,even}^W\right)$ is an isomorphism of algebras.
\end{prop}

On the other hand, we have $Z=F_l(U)^U$, and this subalgebra of invariants may be described explicitly in terms of the Joseph-Letzter decomposition~\eqref{JLdecomp} of $F_l(U)$.  Indeed, the map $J$ yields an identification of $Z$ with the space of matrix elements \beq
\label{invcenter}
\Oq^U = \hc{ \phi\in \Oq \;|\; \phi(Su_1vu_2) = \epsilon(u)\phi(v) \quad\text{for all}\quad u,v\in U }
\eeq
Note that by the $\ad(U_0)$-invariance, we have $\Oq^U\subset \bigoplus_{\lambda\in P} \Oq_{-\lambda,\lambda}$.

\begin{prop}
Suppose that $\phi\in \Oq^U$.  Then
\begin{align}
\label{centerimage}
\xi(\phi)=1\# \lm(S^{-1}\phi)
\end{align}
\end{prop}

\begin{proof}
Using the invariance condition~\eqref{invcenter}, we obtain
\begin{align*}
\left(S^{-1}(\phi_3)\phi_1\otimes S^{-1}\phi_2, u\otimes v\right)&=\phi_3(S^{-1}u_1)\phi_1(u_2)\phi_2(S^{-1}v) =\phi(u_2S^{-1}vS^{-1}u_1)\\
&=\epsilon(u)\phi(S^{-1}v)
\end{align*}
for all $u,v\in U$. By the nondegenerate of the evaluation pairing between $U$ and $\Oq$, we conclude that the equality
$$
S^{-1}(\phi_3)\phi_1\otimes S^{-1}(\phi_2)=1\otimes S^{-1}(\phi)
$$
holds in $\Oq^{\otimes 2}$, and the result follows.
\end{proof}

The restriction of $\widehat{\xi}$ to $Z\subset F_l(U)$ is closely related to the quantum Harish-Chandra homomorphism. Indeed, note that for $\phi=\sum_j \phi_{j}\in \Oq^U$ with $\phi_j\in \Oq_{-\lambda_j,\lambda_j}$, we have
$$
\xi(\phi) = \sum_j \epsilon(\phi_j)(1\# K^{\lambda_j})
$$
On the other hand, if $\phi\in \Oq^U$, we have
$$
(\vartheta\circ J)(\phi) = \sum_j \epsilon(\phi_j)K^{-2\lambda_j}
$$
Introducing the embedding
$$
\upsilon \colon 1\# U_0\rightarrow  U_0, \qquad 1\# K^\mu\longmapsto K^{-2\mu}
$$
we have established

\begin{prop}
The Harish-Chandra map $\vartheta$ may be written as
\begin{align}
\vartheta = \upsilon \circ \widehat{\xi}
\end{align}
\end{prop}

For $\lambda\in P^+$, let $L(\lambda)$ be the finite-dimensional simple $U$ module with highest weight $\lambda$, and let $\chi^\lambda\in \Oq$ be its character. Define functionals
$$
\tau^\lambda(u)=\chi^\lambda(uK^{\rho})
$$
where $\rho$ is the half-sum of positive roots.

\begin{lemma}
The functionals $\tau^\lambda $ are elements of $\Oq^U$.
\end{lemma}
\begin{proof}
The $U$-invariance of $\tau^\lambda$ follows from cyclicity of the trace together with the fact that $S^2(u)=K^{-\rho}uK^{\rho}$ for all $u\in U$.
\end{proof}
Observe that
$$
\xi(\tau^\lambda) = \sum_{\mu\in P} q^{(\mu,\rho)} \dim L(\lambda)^{\mu}K^{\mu}
= \kappa\Br{ \sum_{\mu\in P}\dim L(\lambda)^{\mu}K^{\mu} }
$$
so that $(\kappa^{-1}\circ\xi)(\tau^\lambda)$ coincides with the formal character of $L(\lambda)$.

\begin{cor}
The restriction of $\widehat{\xi}$ to $Z$ gives an isomorphism of algebras $\widehat{\xi} \colon Z \longrightarrow\kappa(U_0^W)$
\end{cor}

\begin{proof}
The orbit sums $\{m(\lambda)=\sum_{\nu\in W\cdot \lambda}K^\nu \;|\; \lambda\in  P^+\}$ form a $\C(q)$-basis for $U_0^W$, and the set of formal characters $\{(\kappa^{-1}\circ\xi)(\tau^\lambda) \;|\; \lambda\in P^+\}$ is triangular with respect this basis under the dominance order on $P^+$.
\end{proof}
Using the homomorphism $\widehat{\xi} \colon Z \rightarrow 1\# U_0$, we may regard $U_0$ as a $Z$-module.  Then, in view of Proposition~\ref{extended-injectivity} in the following section, we have
\begin{cor}
\label{extendedpsi}
The map $\widehat{\xi}$ extends to an embedding of algebras
$$
\widetilde{\xi} \colon F_l(U) \otimes_Z U_0 \longrightarrow \C_q[X], \qquad u\otimes t \longmapsto \widehat{\xi}(u)t
$$
\end{cor}

Based on the isomorphism~\eqref{cartesian} in the classical picture, we make the following
\begin{conjecture}
The map $\widetilde{\xi} \colon F_l(U) \otimes_Z U_0 \longrightarrow \C_q[X]$ is an algebra isomorphism.
\end{conjecture}

\section{Embedding $F_l(U)$ into the Heisenberg double $\Hq$.}
\label{sect-heis}

Composing maps $\xi, \widehat\xi$ with the algebra homomorphism from Corollary~\ref{cqtoheis}, we obtain

\begin{prop}
\label{phi-map}
The maps
\begin{align}
\zeta \colon \OF \longrightarrow \Hq, \qquad &\phi\longmapsto l^+(S^{-1}\phi_3\phi_1)\# \lm(S^{-1}\phi_2) \label{zeta} \\
\widehat\zeta \colon F_l(U) \longrightarrow \Hq, \qquad &\widehat\zeta = \zeta\circ J^{-1} \label{zetahat}
\end{align}
are homomorphisms of algebras.
\end{prop}

\begin{prop}
The image of $\zeta$ is contained in the subalgebra $\Hq^{T_-}$ of $T_-$ invariants.
\end{prop}
\begin{proof}
Suppose that $\psi\in\Oq_{\lambda,\mu}$. Then we may expand
$$
\Delta^2(\psi)=\sum_{\nu_1,\nu_2}\psi_{\lambda,\nu_1}\otimes\psi_{-\nu_1,\nu_2}\otimes\psi_{-\nu_2,\mu}
$$
with $\psi_{\alpha,\beta}\in\Oq_{\alpha,\beta}$. Note that
$$
S^{-1}\psi_{-\nu_2,\mu}\psi_{\lambda,\nu_1}\in\Oq_{\lambda+\mu,\nu_1-\nu_2}
\quad\text{and}\quad
S^{-1}\psi_{-\nu_1,\nu_2}\in \Oq_{\nu_2,-\nu_1}
$$
and recall that $\psi_{\alpha,\beta}(x_\rho)$ is non-zero only if $\rho+\alpha+\beta=0$.  Therefore we have
\begin{align*}
\zeta(\psi)
&= \sum_{\nu_1,\nu_2,\alpha,\beta} (S^{-1}\psi_{-\nu_2,\mu}\psi_{\lambda,\nu_1}) (\Theta_{-\alpha})
\Theta_\alpha K^{-\nu_1+\nu_2}\# (S^{-1}\psi_{-\nu_1,\nu_2}) (\Theta_\beta)\Theta_{-\beta}K^{-\nu_1} \\
&=\sum_{\nu_1,\beta} (S^{-1}\psi_{\nu_1-\beta,\mu}\psi_{\lambda,\nu_1}) (\Theta_{-\lambda-\mu-\beta})
(S^{-1}\psi_{\nu_1,\nu_1-\beta}) (\Theta_{\beta}) \Theta_{\lambda+\mu+\beta} K^{-\beta}\#\Theta_{-\beta}K^{-\nu_1}
\end{align*}
which implies $\zeta(\psi) \in \Hq^{T_-}$.
\end{proof}

Recall the defining representation~\eqref{heisrep} of $\Hq$ on $\Ub$. Pulling this representation back under the algebra homomorphism~\eqref{zetahat}, we obtain an action of the algebra $F_l(U)$ on $\Ub$. In studying this representation, it will be convenient to describe $\Ub$ by means of the surjective homomorphism $l^+ \colon \Oq \rightarrow \Ub$. The following formula is easily deduced from the formula~\eqref{zeta} for $\zeta$, the coquasitriangularity of $r$, and Lemma~\ref{compatible}.

\begin{lemma}
\label{J-action}
The action of $J(\psi)\in F_l(U)$ on $l^+(\varphi)\in\Ub$ induced by $\widehat\zeta$ is given by
$$
\label{inducedaction}
J(\psi)\cdot l^+(\varphi)=r(S^{-1}\psi_3,\varphi_1) \ l^+(S^{-1}\psi_2\varphi_2\psi_1)
$$
\end{lemma}

Since $\zeta(\OF)\subset \Hq^{T_-}$, it follows from Lemma \ref{heisdecomp} that the space $\Ub$ decomposes as an $F_l(U)$-module as
\beq
\label{Fdecomp}
\Ub=\bigoplus_{\lambda\in  P} U^+K^\lambda
\eeq
We will now identify the $F_l(U)$-modules $U^+K^\lambda$. Recall the definition of the contragredient Verma module $M(\mu)^\vee$ for $U$.  Let $\C_{\mu}$ be the one-dimensional $\Ub$-module with basis $w_\mu$ and $\Ub$-module structure defined by
$$
a\cdot w_\mu=\langle a, K^{\mu}\rangle,
$$
which is a slight abuse of notation for $\mu\notin P$. Regard $U$ as a $\Ub$ module via the action $a\cdot u= uS(a)$. Then
$$
\label{contraverma}
M(\mu)^\vee \eqdef \Hom_{\Ub}(U,\C_{\mu})
$$
where $\Hom_{\Ub}$ denotes the restricted (graded) Hom of $\Ub$-modules.  The action of $U$ on $M(\mu)^\vee$ is then given by
$$
(u\cdot\phi)(v)=\phi(Suv).
$$
Note that because of the triangular decomposition of $U$, elements of $M(\mu)^\vee$ are uniquely determined by their values on $\Ubm\subset U$.

\begin{prop}
\label{prop-contra-verma}
The $F_l(U)$-module $U^+K^\lambda$ in the decomposition~\eqref{Fdecomp} is isomorphic to the restriction to $F_l(U)$ of the contragredient Verma module $M(\lambda/2)^\vee$.
\end{prop}

\begin{proof}
Given $a \in U^+K^\lambda$, define an element $\phi_a\in M(\lambda/2)^\vee$ by declaring
$$
\phi_a(y) = \langle aK^{-\lambda/2},y\rangle
$$
for all $y\in \Ubm$. We claim that the map $a\mapsto \phi_a$ is an isomorphism of $F_l(U)$-modules.  By the non-degeneracy of $\langle\cdot,\cdot\rangle$, it is an isomorphism of linear spaces. To show that it respects the $F_l(U)$-module structure, we compute the action of the subalgebras $\Ub$ and $\Ubm$ on $M(\lambda/2)^\vee$.

Suppose first that $z\in \Ubm$, with $Sz\in U^-K^\rho$.  Then for all $y\in \Ubm$ we have
\begin{align*}
(z\cdot \phi_a)(y)&=\phi_a(Szy)
=\langle aK^{-\lambda/2},Szy\rangle
=\langle a_1K^{-\lambda/2},Sz\rangle \langle a_2K^{-\lambda/2},y\rangle
=q^{\frac{1}{2}(\lambda,\rho)}\langle a_1,Sz\rangle\phi_{a_2}(y)
\end{align*}
At the same time, for $b\in U^+K^\rho$, we have
\begin{align*}
(b\cdot \phi_a)(y)
=\phi_a(Sby)
&=\langle Sb_3,y_3\rangle\langle b_1,y_1\rangle\phi_a(y_2Sb_2)
=\langle Sb_3,y_3\rangle\langle b_1,y_1\rangle\langle Sb_2, K^{\lambda/2}\rangle \langle aK^{-\lambda/2}, y_2\rangle\\
&=\langle b_1aK^{-\lambda/2}Sb_3,y\rangle \langle Sb_2, K^{\lambda/2}\rangle
=q^{\frac{1}{2}(\lambda,\rho)} \phi_{b_1aSb_2}(y).
\end{align*}
Here we used formula~\eqref{drinfeldmult} for the product in $U$, together with the homogeneity of the coproduct in $\Ub$. Now given $\psi\in \OF_{\gamma,\mu}$, we compute the action of $J(\psi)=l^+(\psi_1) \lm(\psi_2)$ on $\phi_a\in M(\lambda/2)^\vee$ with the help of Lemma~\ref{J-action}. Note that in the expansion
$$
\Delta(\psi)=\sum_{\nu}\psi_{\gamma,\nu}\otimes\psi_{-\nu,\mu}
$$
we have
$$
lm(S^{-1}\psi_{-\nu,\mu})\in U^-K^\nu \quad\text{and}\quad l^+(\psi_{\gamma,\nu})\in U^+K^{-\nu}.
$$
Then
$$
J(\psi)\cdot\phi_a
=l^+(\psi_1)\cdot \Br{ q^{\frac{1}{2}(\lambda,\nu)}\langle a_1,\lm(S^{-1}\psi_2)\rangle \phi_{a_2} }
=\langle a_1,\lm(S^{-1}\psi_3)\rangle \phi_{l^+(\psi_1)a_2l^+(S^{-1}\psi_2) }.
$$
Therefore taking $a=l^+(\varphi)$, we find
$$
J(\psi)\cdot\phi_{a} = r(S^{-1}\psi_3,\varphi_1)\phi_{l^+(S^{-1}\psi_2\varphi_2\psi_1)} = \phi_{\psi\cdot a}
$$
which shows that the map $a\mapsto \phi_a$ intertwines the two actions of $F_l(U)$.
\end{proof}
\begin{cor}
The homomorphisms $\widehat\zeta$ and $\widehat{\xi}$ are injective.
\end{cor}
\begin{proof}
For any $\lambda\in P^+$, the contragredient Verma module $M(\lambda)^\vee$ contains the finite-dimensional $U$-module $L(\lambda)$ as a submodule. Hence the corollary follows from the fact \cite[5.11]{Jan96} that no non-zero element of $U$ acts by zero in all finite-dimensional representations.
\end{proof}

As in Corollary~\ref{extendedpsi}, we may extend $\widehat\zeta$ to obtain a homomorphism of algebras
\beq
\label{xi-hat}
\widetilde\zeta \colon F_l(U) \otimes_Z U_0 \longra \Hq^{T^-}, \qquad u \otimes t \mapsto \mu(u)t.
\eeq

\begin{prop}
\label{extended-injectivity}
The homomorphism $\widetilde\zeta$ is injective.
\end{prop}

\begin{proof}
Since $U_0 \simeq \C[P]$ we may regard $\widetilde{U} \eqdef F_l(U) \otimes_Z U_0$ as a quasi-coherent sheaf on $\Spec \C[P]$, whose stalk at $\lambda \in \C[P]$ we denote by $\br{\widetilde{U}}_\lambda$. We may similarly regard $\Hq^{T^-}$ as a sheaf over $\Spec \C[P]$ and denote its stalk at $\lambda \in \C[P]$ by $\br{\Hq^{T^-}}_\lambda$. Let $\widetilde\zeta_\lambda \colon \br{\widetilde{U}}_\lambda \longra \br{\Hq^{T^-}}_\lambda$ be the induced map. Then $\ker \widetilde\zeta$ is a subsheaf of $\widetilde{U}$, and $\ker \widetilde\zeta_\lambda$ is its stalk at point $\lambda$. Thus, it is enough to show that $\ker \widetilde\zeta_\lambda = 0$ for any $\lambda$.

Let $\Ic_\lambda \subset \widetilde U$ denote the ideal generated by $\ha{1 \otimes K^\mu - q^{\ha{\lambda,\mu}}}_{\mu \in P}$ and $\Jc_\lambda \subset \Hq^{T_-}$ denote the ideal generated by $\ha{1 \# K^\mu - q^{\ha{\lambda,\mu}}}_{\mu \in P}$. Let $U^\lambda$ be the quotient of $U$ by the central character of the Verma module of weight $\lambda$. Note, that $\widetilde U / \Ic_\lambda \simeq U^\lambda$. Set $\Hq^\lambda \eqdef \Hq^{T_-}/\Jc_\lambda$ and let $\widehat\zeta^\lambda \colon U^\lambda \longra \Hq^\lambda$ be the induced homomorphism. By quantum Duflo theorem, we know that $U^\lambda$ acts faithfully on the Verma module $M(\lambda)$. In view of Proposition~\ref{prop-contra-verma} and the existence of a nondegenerate pairing between a Verma module and the corresponding contragredient Verma module, we obtain $\ker \widehat\zeta^\lambda = 0$.

Now, let $\C[P]_\lambda$ denote the local ring at $\lambda$ and $\mgt_\lambda$ be its maximal ideal. Then one has
$$
\U^\lambda = \br{\widetilde U}_\lambda / \mgt_\lambda \br{\widetilde U}_\lambda
\qquad\text{and}\qquad
\Hq^\lambda = \br{\Hq^{T_-}}_\lambda / \mgt_\lambda \br{\Hq^{T_-}}_\lambda,
$$
so that
$$
\mgt_\lambda \ker \widetilde\zeta_\lambda = \ker \widetilde\zeta_\lambda.
$$
At this point the Proposition would from Nakayama's lemma if $\ker \widetilde\zeta_\lambda$ were a finitely-generated $\C[P]_\lambda$ module. Therefore, it remains to filter $\ker \widetilde\zeta_\lambda$ by finitely generated submodules. There is a natural filtration on $\br{\widetilde U}_\lambda$ (by the sum of modulus of exponents in the Poincar\'e-Birkhoff-Witt basis), so let $\ker_n \widetilde\zeta_\lambda$ denote the intersection of the $n$-th filtered component with $\ker \widetilde\zeta_\lambda$. Then the submodules $\ker_n \widetilde\zeta_\lambda$ are finitely generated (as submodules of a finitely generated module over a Noetherian ring) and deliver the required filtration on $\ker \widetilde\zeta_\lambda$.
\end{proof}

\section{The $R$-twisted quantum coordinate ring}
\label{sect-R-twist}

In this section we introduce the $R$-twist $\OR$ of the quantum coordinate ring $\Oq$, and explain its relation with the Heisenberg double $\Hq$.

\subsection{The Heisenberg double and $\OR$ }

\begin{prop} The following formula defines an associative product $\bullet_R$ in $\Oq$
\beq
\label{prod-R}
\phi\bullet_R\psi = r(\phi_{1},\psi_{1})\phi_{2}\psi_{2}
\eeq
\end{prop}
\begin{proof}
This follows straightforwardly from the co-quasitriangularity properties~\eqref{rform-prop} of the universal $r$-form.
\end{proof}

\begin{defn}
We define $\OR$ to be the associative algebra with multiplication defined by~\eqref{prod-R}.
\end{defn}

\begin{prop}
The map $I$ given by~\eqref{I} defines an embedding of algebras
$$
I \colon \OR \longrightarrow \Hq.
$$
\end{prop}
\begin{proof}
That $I$ is injective follows from the injectivity of the map $J=m\circ I$.   To prove that $I$ is a homomorphism of algebras, we compute
$$
I(\phi\bullet_R\psi)
=r(\phi_{1},\psi_{1}) I(\phi_{2}\psi_{2})
=r(\phi_{1},\psi_{1}) l^+(\phi_{2}\psi_{2}) \# \lm(\phi_{3}\psi_{3})
$$
On the other hand,  in $\Hq$ we have
\begin{align*}
I(\phi) \cdot I(\psi)
&= \left(l^+(\phi_1)\#\lm(\phi_2)\right) \cdot \left(l^+(\psi_1)\#\lm(\psi_2)\right) \\
&= \ha{ l^+(\psi_1)_2, \lm(\phi_2)_2 } l^+(\phi_1) l^+(\psi_1)_1 \# \lm(\phi_2)_1 \lm(\psi_2) \\
&= \ha{ l^+(\psi_2), \lm(\phi_2) } l^+(\phi_1) l^+(\psi_1) \# \lm(\phi_3) \lm(\psi_3) \\
&= r(\phi_2,\psi_2)l^+(\psi_1\phi_1) \# \lm(\phi_3\psi_3) \\
&= r(\phi_1,\psi_1)l^+(\phi_2\psi_2) \# \lm(\phi_3\psi_3) \\
&= I(\phi\bullet_R\psi)
\end{align*}
\end{proof}

\begin{prop}
The image $I(\OR)\subset \Hq$ is contained in the subalgebra $\Hq^{T^c}$ of $T^c$-invariants.
\end{prop}

\begin{proof}
Suppose that $\psi\in\Oq_{\lambda,\mu}$, and
$$
\Delta(\psi)=\sum_\nu\psi_{\lambda,\nu}\otimes\psi_{-\nu,\mu}
$$
Then
$$
I(\psi)=\sum\psi_{\lambda,\nu}(\Theta_{-\alpha})\psi_{-\nu,\mu}(\Theta_\beta)\Theta_\alpha K^{-\nu}\#\Theta_{-\beta}K^{-\mu}
$$
The only non-zero terms in the sum must have $\beta+\mu-\nu=0$, $\lambda+\nu-\alpha=0$. Hence we find
$$
I(\psi) = \sum \psi_{\lambda,\mu+\beta}(\Theta_{-\lambda-\mu-\beta})\psi_{-\mu-\beta,\mu}(\Theta_\beta)\Theta_{\lambda+\mu+\beta}
K^{-\mu-\beta}\#\Theta_{-\beta}K^{-\mu} \in \Hq^{T_c}
$$
\end{proof}

Although $I:\OR\rightarrow \Hq^{T_c}$ is an embedding, it is not surjective.  In order to obtain an isomorphism, we must localize at certain elements of $\OR$. We define elements $\phi^\pm_i\in\Oq$ by
\begin{align*}
\phi_i^+ &= (q_i^{-1}-q_i)^{-1} \Delta^{\omega_i}_{s_i,1}, \\
\phi^-_i &= (q_i^{-1}-q_i)^{-1} \Delta^{\omega_i}_{1,s_i}.
\end{align*}

\begin{lemma}
The following equalities hold
\label{chevgenI}
\begin{align*}
& I(\Delta^{\omega_i})=K^{-\omega_i}\# K^{-\omega_i},\\
& I(\phi^+_i)=E_iK^{-\omega_i}\# K^{-\omega_i},\\
& I(\phi^-_i)=K^{-\omega_i}\# F_i K^{\alpha_i-\omega_i}.
\end{align*}
\end{lemma}

\begin{proof}
One can see that
$$
\phi^+_i = (1-q_i^{-2})^{-1}\ad^*_r(E_i)(\Delta^{\omega_i})
\qquad\text{and}\qquad
\phi^-_i = (q_i^{-1}-q_i)^{-1}\ad^*_r(F_i)(\Delta^{\omega_i}).
$$
The rest of the proof is a straightforward calculation using the $U$-equivariance of $J$.
\end{proof}

\begin{prop}
\label{i-map-prop}
The algebra $\Hq^{T_c}$ is generated by $I\hr{\OR}$ together with the elements
$$
K^{\omega_i}\# K^{\omega_i}=q^{(\omega_i,\omega_i)} I\hr{\Delta^{\omega_i}}^{-1}.
$$
Hence, we have the isomorphism
\beq
\label{i-map}
I \colon \OR[\hr{\Delta^{\omega_i}}^{-1}]_{i=1}^r \longrightarrow \Hq^{T_c}
\eeq
\end{prop}

\begin{proof}
Existence of the map and its injectivity follow from the fact that $I(\Delta^{\omega_i})$ is invertible in $\Hq$, with inverse given by
$$
I(\Delta^{\omega_i})^{-1}=q^{-(\omega_i,\omega_i)}K^{\omega_i}\# K^{\omega_i}
$$
The surjectivity follows from Lemma~\ref{chevgenI} together with the description~\eqref{tcinvts} of $\Hq^{T_c}$.
\end{proof}

Set
$$
\Oc_q[G^\circ] \;\eqdef\; \Oq[(\Delta^{\omega_i})^{-1}]_{i=1}^r
$$
and let
$$
\Oc_q[G^\circ/H] \;\eqdef\; \bc{ \phi\in \Oc_q[G^\circ] \;|\; (K^\lambda\otimes 1) \cdot \phi=\phi \quad\text{for any}\quad \lambda\in P }.
$$
be the subalgebra of $U_0$-invariants in $\Oc_q[G^\circ]$ under the coregular action defined by~\eqref{lractions}.

\begin{cor}
\label{i-map-cor}
The restriction of the map~\eqref{i-map}
$$
I \colon {}^R\Oc_q[G^\circ/H] \longrightarrow \Hq^{\Tbb}
$$
is an isomorphism of algebras.
\end{cor}

\subsection{Images of the Chevalley generators under $\widehat\zeta$}
By Lemma~\ref{chevgenI}, it suffices to calculate $\zeta(\Delta_i), \zeta(\phi_i^\pm)$.  First, suppose that $\phi\in \Oq_{\lambda,\mu}$ satisfies $\phi(xu) = \epsilon(x)\phi(u)$ for all $x\in U^-$, $u \in U$. Then we have $l^+(\phi)=\epsilon(\phi)K^{\lambda}$, so
$$
\zeta(\phi)=(1\# K^{\lambda})I(S^{-1}\phi)
$$
which implies
\begin{align*}
\zeta(\Delta_i)&=q^{-(\omega_i,\omega_i)}I(\Delta^{\omega_i})\cdot I(S^{-1}\Delta^{\omega_i})t^{\omega_i}, \\
\zeta(\phi^-_i)&=q^{-(\omega_i,\omega_i)}I(\Delta^{\omega_i})\cdot I(S^{-1}\phi^-_i)t^{\omega_i},
\end{align*}
where $t^{\lambda}$ stands for $1 \# K^\lambda$.

In order to calculate $\zeta(\phi^+_i)$, suppose that $\phi\in\Oq_{\lambda,\mu}$ satisfies $\phi(ua) = \epsilon(a)\phi(u)$ for all $a\in U^+$ and $u\in U$. Then we have $J(\phi)=l^+(\phi)K^{-\mu}$, and hence
$$
\label{coproducttrick}
\Delta_U(J(\phi))=l^+(\phi_1)K^{-\mu}\otimes J(\phi_2),
$$
where $\Delta_U$ denotes the comultiplication in $U$.
In turn, this implies
$$
\zeta(\phi_i^+)
= q^{-(\omega_i,\omega_i)}\Br{I(\phi^+_i)\cdot I(S^{-1}\Delta^{\omega_i})t^{\omega_i}
+ q_iI(\Delta^{\omega_i})\cdot I(\Delta^{\alpha_i})^{-1}\cdot I(S^{-1}\phi^+_i)t^{\omega_i-\alpha_i}}.
$$

\begin{cor}
\label{chev-gen}
We have
\begin{align*}
&\widehat\zeta(K^{-2\omega_i}) = q^{-(\omega_i,\omega_i)}I(\Delta^{\omega_i})\cdot I(S^{-1}\Delta^{\omega_i})t^{\omega_i}, \\
&\widehat\zeta(F_iK^{\alpha_i-2\omega_i}) = q^{-(\omega_i,\omega_i)}q_i^{-1}I(\Delta^{\omega_i})\cdot I(S^{-1}\phi^-_i)t^{\omega_i}, \\
&\widehat\zeta(E_iK^{-2\omega_i}) = q^{-(\omega_i,\omega_i)}\Br{I(\phi^+_i)\cdot I(S^{-1}\Delta^{\omega_i})t^{\omega_i}
+ q_iI(\Delta^{\omega_i})\cdot I(\Delta^{\alpha_i})^{-1}\cdot I(S^{-1}\phi^+_i)t^{\omega_i-\alpha_i}}.
\end{align*}
\end{cor}

\subsection{An isomorphism between $\OR$ and $\Oq$}
We now explain how to use the quantum Weyl group to construct an isomorphism between $\OR$ and $\Oq$.  Recall the element $Y$ defined by \eqref{y-element}.  Then the identity \eqref{cocycle}implies 
the following proposition.  
\begin{prop}
\label{prop-R-twist-iso}
The element $Y$ defines an isomorphism of algebras
$$
\label{weylsio}
\iota_{Y} \colon \Oq \longrightarrow \OR, \qquad \phi \longmapsto \langle Y,\phi_1\rangle \phi_2
$$
\end{prop}

\begin{proof}
Using the relation~\eqref{cocycle}, we compute
\begin{align*}
\iota_{Y}(\phi)\bullet_R\iota_{Y}(\psi)&=\langle Y,\phi_1\rangle\langle Y,\psi_1\rangle \phi_1\bullet_R\psi_2
=\langle Y,\phi_1\rangle\langle Y,\psi_1\rangle r(\phi_2,\psi_2)\phi_3\psi_3\\
&=\langle Y,\phi_1\rangle\langle Y,\psi_1\rangle \langle (Y^{-1} \otimes Y^{-1})\Delta(Y),\phi_2\otimes\psi_2\rangle\phi_3\psi_3\\
&=\langle \Delta(Y),\phi_1\otimes\psi_1\rangle\phi_2\psi_2
=\langle Y,\phi_1\psi_1\rangle \phi_2\psi_2
=\iota_{Y}(\phi\psi).
\end{align*}
\end{proof}

\begin{defn}
Let $\theta$ be the Dynkin diagram automorphism such that $w_0s_i = s_{\theta(i)}w_0$ holds for all simple reflections $s_i$.
\end{defn}

Using the definition of $Y$ one obtains the following explicit formulas for $\iota_Y$ in terms of generalized minors.

\begin{lemma}
\label{iota-explicit}
One has
\begin{align*}
(\iota_Y)^{-1}(\Delta^{\omega_i}) &= q^{(\omega_i,\rho+\omega_i/2)}\Delta^{\omega_i}_{w_0,1} \\
(\iota_Y)^{-1}(S^{-1}\Delta^{\omega_i}) &= (-1)^{\ha{ 2\omega_i,\rho^\vee}}q^{(\omega_i,\rho+\omega_i/2)}\Delta^{\omega_{\theta(i)}}_{1,w_0} \\
(q_i^{-1}-q_i)(\iota_Y)^{-1}(\phi_i^+) &= -q^{(\omega_i,\rho+\omega_i/2)}\Delta^{\omega_i}_{w_0s_i,1}\\
(q_i^{-1}-q_i)(\iota_Y)^{-1}(S^{-1}\phi_i^+) &= (-1)^{\ha{2\omega_i,\rho^\vee}+1}q_i^{-1}q^{(\omega_i,\rho+\omega_i/2)}\Delta^{\omega_{\theta(i)}}_{1,s_iw_0} \\
(q_i^{-1}-q_i)(\iota_Y)^{-1}(S^{-1}\phi_i^-) &= (-1)^{\ha{2\omega_i,\rho^\vee}}q_iq^{(\omega_i,\rho+\omega_i/2)}\Delta^{\omega_{\theta(i)}}_{s_{\theta(i)},w_0}
\end{align*}
\end{lemma}

\begin{cor}
\label{localized-twist-prop}
The map $\iota_Y$ establishes an isomorphism between the localizations
$$
\label{localizedtwist}
\iota_Y \colon \Oq[(\Delta^{\omega_i}_{w_0,1})^{-1}]_{i=1,\ldots,r} \longrightarrow \OR[(\Delta^{\omega_i})^{-1}]_{i=1,\ldots,r}.
$$
\end{cor}
As explained in \cite{BZ05}, the algebra $\Oq[(\Delta^{\omega_i}_{w_0,1})^{-1}]_{i=1,\ldots,r}$ can be regarded as the quantum coordinate ring $\mathcal{O}_q[G^{w_0}]$ of the big open Bruhat cell $G^{w_0}=B_+w_0B_+\subset G$.

\section{Main results}
\label{sect-main}

Let us introduce the notation
$$
\Oc_q[G^{w_0}/H] \eqdef \bc{ \phi\in \Oc_q[G^{w_0}] \;|\; (K^\lambda\otimes 1) \cdot \phi=\phi \quad\text{for any}\quad \lambda\in P }
$$
for the subalgebra of $U_0$-invariants in $ \Oc_q[G^{w_0}]$ under the coregular action defined by~\eqref{lractions}.
By Corollary~\ref{i-map-cor} and Corollary~\ref{localized-twist-prop} the map
\beq
\label{blah}
(\iota_Y^{-1} \circ I^{-1}) \otimes \id \colon \Hq^\Tbb \otimes T \simeq \Hq^{T-} \longrightarrow \mathcal{O}_q[G^{w_0}/H] \otimes T
\eeq
is an isomorphism of algebras. Combining this isomorphism with Corollary~\ref{heis-decomposition}, we arrive at
\begin{theorem}
\label{thm-main}
The map $\Phi$ obtained by composing the homomorphism $\widetilde\zeta$ defined in~\eqref{xi-hat} with the isomorphism~\eqref{blah} is an embedding of algebras
$$
\Phi \colon F_l(U) \otimes_Z U_0 \longrightarrow \mathcal{O}_q[G^{w_0}/H] \otimes T.
$$
\end{theorem}

\begin{remark}
\label{rem-main}
Note that by Corollary~\ref{chev-gen}, in order to extend the homomorphism $\Phi$ to the entire quantum group $\Uq$, we must localize further by inverting the products $\Delta^{\omega_i}_{w_0,1}\Delta^{\omega_{\theta(i)}}_{1,w_0}$ for all $i=1,\cdots r$. Hence the target of the homomorphism becomes $\mathcal{O}_q[G^{w_0,w_0}/H]$, the quantum coordinate ring of the reduced big double Bruhat cell in $G$. In fact, we must also adjoin the square roots $\hr{\Delta^{\omega_i}_{w_0,1}\Delta^{\omega_{\theta(i)}}_{1,w_0}}^{1/2}$ to $\mathcal{O}_q[G^{w_0,w_0}/H]$, although this poses no difficulties. This phenomenon is related to the fact that the maps $\eta_i \colon G^*\rightarrow G_*$ in~\eqref{eta1}, \eqref{eta2}, while local diffeomorphisms, are in fact $2^r$-fold coverings.
\end{remark}

\begin{notation}
$\mathcal{O}'_q[G^{w_0,w_0}/H]$ denotes the algebra obtained by adjoining $\hr{\Delta^{\omega_i}_{w_0,1}\Delta^{\omega_{\theta(i)}}_{1,w_0}}^{1/2}$ for $i = 1 \sco r$ to $\mathcal{O}_q[G^{w_0,w_0}/H]$. Similarly, $T'$ stands for $\C[P/2] \supset T$.
\end{notation}

\begin{cor}
Let $\chi\colon T \rightarrow \C$ be a character of the torus $T$. Denote by the same letter the induced character of the center $Z \subset \Uq$ coming from the embedding $\widehat\xi|_Z \colon Z \hookrightarrow T$. Then $\Phi$ extends to an embedding $\Phi'\colon \Uq \to  \Oc'_q[G^{w_0,w_0}/H] \otimes T'$ such that the following diagram commutes
$$
\xymatrix{
\Uq \ar[r]^{\hspace{-30pt}\Phi'} \ar[d] & \Oc'_q[G^{w_0,w_0}/H] \otimes T' \ar[d]^{\id\otimes\chi} \\
\Uq/\Ic_\chi \ar[r] & \Oc'_q[G^{w_0,w_0}/H]
}
$$
where $\Ic_\chi$ is the ideal generated by the kernel of $\chi$.
\end{cor}

\begin{cor}
\label{phi-prime}
One has the following explicit formulas for $\Phi'$
\begin{align*}
&\Phi'(K^{-2\omega_i}) = (-1)^{\langle 2\omega_i, \rho^\vee\rangle}q^{2(\omega_i,\rho)}\Delta^{\omega_i}_{w_0,1}\Delta^{\omega_{\theta(i)}}_{1,w_0}t^{\omega_i},\\
&\Phi'(\widehat{F}_iK^{\alpha_i}) = q_i\Delta^{\omega_{\theta(i)}}_{s_{\theta(i)},w_0}
\hr{\Delta^{\omega_{\theta(i)}}_{1,w_0}}^{-1}, \\
&\Phi'(\widehat{E}_i) = -\left( \Delta^{\omega_i}_{w_0s_i,1}\hr{\Delta^{\omega_{i}}_{1,w_0}}^{-1} +  q_i^{-1}\Delta^{\omega_{\theta(i)}}_{1,w_0s_{\theta(i)}}\hr{\Delta^{\omega_{\theta(i)}}_{1,w_0}}^{-1}\left(\Delta^{\alpha_i}_{w_0,1}\right)^{-1}t^{-2\omega_i}\right),
\end{align*}
where $\widehat E_i = (q_i^{-1} - q_i)E_i$ and $\widehat F_i = (q_i^{-1} - q_i)F_i$.
\end{cor}
\begin{proof}
This follows from combining Corollary~\ref{chev-gen} with Lemma~\ref{iota-explicit}.  
\end{proof}

We end this section with the following conjecture based on the classical isomorphism~\eqref{cartesian} and Proposition~\ref{birational-chain}.
\begin{conjecture}
We have an isomorphism of non-commutative fraction fields
$$
\Frac\hr{F_l(U) \otimes_Z U_0} = \Frac \left(\mathcal{O}_q[G/H] \otimes T \right).
$$
In particular, $\Frac\hr{F_l(U) \otimes_Z U_0}$ coincides with a non-commutative fraction field of a quantum torus algebra, i.e. the quantum Gelfand-Kirillov property holds for $F_l(U) \otimes_Z U_0$.
\end{conjecture}

\section{Example for $\g=\sl_2$}
\label{sect-ex}

We conclude by providing a detailed example of our construction for the case $\g=\sl_2$.  Let us write $E,F,K^{1/2}$ for the generators of the simply-connected form of $U_q(\sl_2)$.  Recall that the fundamental representation of $U_q(\sl_2)$ on $\bC^2$ is determined by
$$
E \mapsto
\begin{pmatrix}
0 & 1 \\
0 & 0
\end{pmatrix},
\qquad
F \mapsto
\begin{pmatrix}
0 & 0 \\
1 & 0
\end{pmatrix},
\qquad
K^{1/2} \mapsto
\begin{pmatrix}
q^{1/2} & 0 \\
0 & q^{-1/2}
\end{pmatrix}.
$$
The Hopf algebra $\Oc_q(SL_2)$ is generated by the matrix coefficients of the fundamental representation. More explicitly, $\Oc_q(SL_2)$ has generators $\ha{x_{11}, x_{12}, x_{21}, x_{22}}$ subject to the relations
\begin{align*}
x_{11}x_{12} &= q x_{12}x_{11} & x_{12}x_{22} &= q x_{22}x_{12} & x_{12}x_{21} &= x_{21}x_{12} \\
x_{11}x_{21} &= q x_{21}x_{11} & x_{21}x_{22} &= q x_{22}x_{21} & [x_{11},x_{22}] &= (q-q^{-1})x_{12}x_{21}
\end{align*}
as well as the quantum determinant relation
$$
x_{11}x_{22} - qx_{12}x_{21} = 1.
$$
The coalgebra structure of $\Oc_q(SL_2)$ is given by
$$
\Delta(x_{ij}) = x_{i1} \otimes x_{1j} + x_{i2} \otimes x_{2j} \qquad\text{and}\qquad \epsilon(x_{ij}) = \delta_{ij}
$$
while the antipode is given by
$$
S(x_{11}) = x_{22}, \qquad S(x_{12}) = -q^{-1}x_{12}, \qquad S(x_{21}) = -qx_{21}, \qquad S(x_{22}) = x_{11}.
$$
The quantum coordinate ring of the big Bruhat cell $Bw_0B\subset SL_2$ is
$$
\Oc_q[SL_2^{w_0}]=\Oc_q[SL_2][x_{21}^{-1}]
$$
while the quantum coordinate ring of the big double Bruhat cell $Bw_0B\cap B_- w_0 B_-\subset SL_2$ is given by
$$
\Oc_q[SL_2^{w_0,w_0}]=\Oc_q[SL_2][x_{12}^{-1}x_{21}^{-1}].
$$
The quantum coordinate ring of the reduced big double Bruhat cell $\Oc_q[SL_2^{w_0,w_0}/H]\otimes T$ embeds into the quantum torus algebra
$$
\mathcal{A}=\bC\langle u^{\pm 1},v^{\pm 1}, z^{\pm 1}\rangle\big\slash (uv=q^{2}vu, zu=uz,zv=vz)
$$
via the identification
$$
u = -q^2x_{22}x_{21}, \qquad  v = -q^{-1}x_{12}^{-1}x_{21}^{-1}, \qquad z = t.
$$

As in Corollary~\ref{phi-prime}, we introduce the normalized generators of $U_q(\sl_2)$
$$
\wh E = (q^{-1}-q)E \qquad\text{and}\qquad \wh F = (q^{-1}-q)F.
$$
Then the values of the $l$-operators on the matrix coefficients $x_{ij}$ are easily computed to be
\begin{align*}
l^+(x_{11}) &= K^{-1/2}  & \lm(x_{11}) &= K^{-1/2} \\
l^+(x_{12}) &= 0 & \lm(x_{12}) &= \wh FK^{1/2} \\
l^+(x_{21}) &= \wh EK^{-1/2} & \lm(x_{21}) &= 0\\
l^+(x_{22}) &= K^{1/2} & \lm(x_{22}) &= K^{1/2}
\end{align*}

It follows that the isomorphism $J \colon \Oc_q(SL_2) \longra F_l(U_q(\sl_2))$ is given by
$$
J(x_{11}) = K^{-1}, \qquad J(x_{12}) = q \wh F, \qquad J(x_{21}) = \wh E K^{-1}, \qquad J(x_{22}) = K + q \wh E \wh F.
$$
The homomorphism $\zeta \colon \Oc_q(SL_2) \longra \mathcal H_q^{T_-}$ takes the form
\begin{align*}
\zeta(x_{11}) &= \br{1 + q^{-1} \wh EK^{-1} \# \wh F} \cdot t \\
\zeta(x_{12}) &= -q \br{K^{-1} \# \wh F} \cdot t \\
\zeta(x_{21}) &= \br{\wh E \# 1 + q^{-1} \wh E^2 K^{-1} \# \wh F} \cdot t - \br{\wh E \# 1} \cdot t^{-1} \\
\zeta(x_{22}) &= -q \br{\wh EK^{-1} \# \wh F} \cdot t + t^{-1}
\end{align*}
where we write $t=1 \# K^{1/2}\in \mathcal{H}_q^{T_-}$.

The isomorphism $\iota_Y \colon \Oc_q(SL_2) \longra {}^R\Oc_q(SL_2)$ is given by
\begin{align*}
\iota_Y(x_{11}) &= -q^{-3/4}x_{21}, \qquad \iota_Y(x_{21}) = q^{-3/4}x_{11}, \\
\iota_Y(x_{12}) &= -q^{-3/4}x_{22}, \qquad \iota_Y(x_{22}) = q^{-3/4}x_{12}.
\end{align*}
and the isomorphism $I \colon \Oc^R_q(SL_2)[\Delta_i^{-1}]_{i=1}^r \longra \Hq^{T_c}$ is
\begin{align*}
I(x_{11}) &= K^{-1/2} \# K^{-1/2}      & I(x_{21}) &= \wh E K^{-1/2} \# K^{-1/2} \\
I(x_{12}) &= K^{-1/2} \# \wh F K^{1/2} & I(x_{22}) &= K^{1/2} \# K^{1/2} + \wh E K^{-1/2} \# \wh F K^{1/2}
\end{align*}

The algebra embedding $\Phi  \colon F_l(U_q(\sl_2)) \longra \Oc_q[SL_2^{w_0}/H] \otimes T$ in Theorem~\ref{thm-main} takes the form
\beq
\label{bl}
\begin{aligned}
&K^{-1} \mapsto -qx_{12}x_{21}t, \\
&\wh F \mapsto -q^2 x_{22}x_{21}t, \\
&\wh EK^{-1} \mapsto qx_{11}x_{12}t + x_{11}x_{21}^{-1}t^{-1}.
\end{aligned}
\eeq
As explained in Remark~\ref{rem-main}, in order to embed $U_q(\sl_2)$ we must localize further at $x_{12}x_{21}$ and adjoin $(x_{12}x_{21})^{1/2}$.
Therefore let $\mathcal A'$ be the quantum torus algebra obtained from $\mathcal{A}$ by adjoining the elements $v^{1/2}$ and $z^{1/2}$. Then we obtain the following quantum torus algebra realization of $U_q(\sl_2)$:
$$
\Phi'  \colon U_q(\sl_2)\longrightarrow \mathcal{A}'
$$
$$
K^{1/2} \mapsto v^{1/2}z^{-1/2}, \qquad \wh F \mapsto uz, \qquad \wh E \mapsto z^{-1}u^{-1}(qv^{1/2}-q^{-1}v^{-1/2})(v^{-1/2}z-v^{1/2}z^{-1}).
$$

\bibliographystyle{alpha}

\end{document}